\documentclass[11pt,leqno]{article}

\usepackage[active]{srcltx}

\usepackage{amsfonts,latexsym,amsmath,amssymb,amsthm}
\usepackage{fullpage}
\usepackage{comment}
\newtheorem{theorem}{Theorem}[section]
\newtheorem{lemma}[theorem]{Lemma}

\newtheorem{proposition}[theorem]{Proposition}
\newtheorem{corollary}[theorem]{Corollary}
\newtheorem{remark}[theorem]{Remark}

\theoremstyle{definition}

\theoremstyle{remark}

\newtheorem*{note*}{Note}

\numberwithin{equation}{section}

\makeatletter
\newcommand{\rank}{\mathop{\operator@font rank}}
\newcommand{\conv}{\mathop{\operator@font conv}}
\newcommand{\vol}{\mathop{\operator@font vol}}
\newcommand{\onetagright}{\tagsleft@false}
\makeatother

\newcommand{\ls}{\leqslant}
\newcommand{\gr}{\geqslant}

\usepackage{color}

\usepackage{hyperref}

\begin{document}
\small

\title{\bf On a version of the slicing problem for the surface area of convex bodies}

\medskip

\author{Silouanos Brazitikos and Dimitris--Marios Liakopoulos}

\date{}

\maketitle

\begin{abstract}\footnotesize
We study the slicing inequality for the surface area instead of volume. This is the question whether there exists a constant $\alpha_n$ depending
(or not) on the dimension $n$ so that
\begin{equation*}S(K)\ls\alpha_n|K|^{\frac{1}{n}}\max_{\xi\in S^{n-1}}S(K\cap\xi^{\perp })\end{equation*}
where $S$ denotes surface area and $|\cdot |$ denotes volume. For any fixed dimension we provide a negative answer to this question, as well
as to a weaker version in which sections are replaced by projections onto hyperplanes. We also study the same problem for sections and projections
of lower dimension and for all the quermassintegrals of a convex body. Starting from these questions, we also introduce a number of natural parameters relating volume and surface area, and provide optimal upper and lower bounds for them. Finally, we show that, in contrast to the previous negative
results, a variant of the problem which arises naturally from the surface area version of the equivalence of the isomorphic Busemann--Petty problem with the slicing problem has an affirmative answer.
\end{abstract}

\section{Introduction}

In this article we study the question whether it is possible to have a version of the slicing inequality for the surface area instead of volume.
More precisely, the question (which has been formulated by Koldobsky \cite{Koldobsky-private}) can be stated as follows: Is it true
that there exists a constant $\alpha_n$ depending (or not) on the dimension $n$ so that
\begin{equation}\label{eq:slicing-surface-1}S(K)\ls\alpha_n|K|^{\frac{1}{n}}\max_{\xi\in S^{n-1}}S(K\cap\xi^{\perp })\end{equation}
for every centrally symmetric convex body $K$ in ${\mathbb R}^n$? Here, $S(A)$ denotes surface area and $|A|$ denotes volume of
a convex body in the appropriate dimension, and $\xi^{\perp }=\{x\in {\mathbb R}^n:\ \langle x,\xi\rangle =0\}$ is the $(n-1)$-dimensional
subspace orthogonal to $\xi\in S^{n-1}$. A lower dimensional slicing problem may be also formulated; for any $2\ls k\ls n-1$ one may ask for a
constant $\alpha_{n,k}$ such that
\begin{equation}\label{eq:slicing-surface-lower}S(K)\ls\alpha_{n,k}^k|K|^{\frac{k}{n}}\max_{H\in G_{n,n-k}}S(K\cap H)\end{equation}
for every centrally symmetric convex body $K$ in ${\mathbb R}^n$, where $G_{n,s}$ is the Grassmann
manifold of all $s$-dimensional subspaces of ${\mathbb R}^n$. Moreover, one may replace surface area by any other quermassintegral
and pose the corresponding question (see Section~2 for definitions and background information).

\smallskip

{\it The slicing problem and its variants.} In order to put our question into context we start by recalling
the classical Busemann-Petty problem \cite{BP}: Let $K$ and $D$ be centrally symmetric convex bodies in ${\mathbb R}^n$ and assume that
$|K\cap \xi^\perp |\ls |D\cap\xi^\perp |$ for every $\xi\in S^{n-1}.$ Is it then true that $|K|\ls |D|?$ It is known that the answer
is affirmative if $n\ls 4$ and negative if $n\gr 5;$ see
\cite{Gardner-book} and \cite{Koldobsky-book} for the history and the solution of the problem.
An isomorphic version of the Busemann-Petty problem was introduced in \cite{MP}. Does there exist an absolute constant $C_1$
so that for any dimension $n$ and any pair of centrally symmetric convex bodies $K$ and $D$ in ${\mathbb R}^n$ satisfying
$|K\cap \xi^\perp|\ls |D\cap \xi^\perp|$ for all $\xi\in S^{n-1}$
we have that $ |K|\ls C_1 |D| ?$ The isomorphic Busemann-Petty problem is equivalent to the slicing problem
which asks if there exists an absolute constant $C_2>0$ such that for every $n\gr 2$ and every
convex body $K$ in ${\mathbb R}^n$ with barycenter at the origin (we call these convex bodies centered)
one has
\begin{equation*}|K|^{\frac{n-1}{n}}\ls C_2\,\max_{\xi\in S^{n-1}}\,|K\cap \xi^{\perp }|.\end{equation*}
It is well-known that this problem is equivalent to the question if there exists an absolute constant $C_3>0$ such that
\begin{equation*}L_n:= \max\{ L_K:K\ \hbox{is isotropic in}\ {\mathbb R}^n\}\ls C_3\end{equation*}
for all $n\gr 1$, where $L_K$ is the isotropic constant of $K$. Bourgain proved in \cite{Bourgain-1991} that $L_n\ls c_1\sqrt[4]{n}\log\! n$, and Klartag \cite{Klartag-2006} improved this bound to $L_n\ls c_2\sqrt[4]{n}$. A breakthrough on this problem has been recently announced by Y.~Chen \cite{YChen}; from his results it follows that $L_n\ls C\cdot\exp (c\sqrt{\log n}\cdot\sqrt{\log (\log n)})=o(n^{\epsilon })$ for any $\epsilon >0$ as the dimension $n$ grows to infinity. From the equivalence of the two questions it follows that
\begin{equation*}|K|^{\frac{n-1}{n}}\ls c_3L_n\,\max_{\xi\in S^{n-1}}\,|K\cap \xi^{\perp }|\end{equation*}
for every centered convex body $K$ in ${\mathbb R}^n$. The lower dimensional slicing problem can be posed in the
following way: Let $1\ls k\ls n-1$ and let $\alpha_{n,k}$
be the smallest positive constant $\alpha >0$ such that, for every centered convex body $K$ in ${\mathbb R}^n$,
\begin{equation*}|K|^{\frac{n-k}{n}}\ls \alpha^k\max_{H\in G_{n,n-k}}|K\cap H|.\end{equation*}
Then the question is if there exists an absolute constant $C_4>0$ such that $\alpha_{n,k}\ls C_4$ for all $n$ and $k$.

The slicing problem can be posed for a general measure in place of volume. Let $g$ be a locally integrable non-negative function on ${\mathbb R}^n$.
For every Borel subset $B\subseteq {\mathbb R}^n$ we define
\begin{equation*}\mu (B)=\int_Bg(x)dx,\end{equation*}
where, if $B\subseteq H$ for some subspace $H\in G_{n,s}$, $1\ls s\ls n-1$, integration is understood with respect to the $s$-dimensional Lebesgue
measure on $H$. Then, for any $1\ls k\ls n-1$ one may define $\alpha_{n,k}(\mu )$ as the smallest constant $\alpha >0$ with the following property: For
every centered convex body $K$ in ${\mathbb R}^n$ one has
\begin{equation*}\mu (K)\ls \alpha^k\,|K|^{\frac{k}{n}}\,\max_{H\in G_{n,n-k}}\mu (K\cap H).\end{equation*}
Koldobsky proved in \cite{Koldobsky-Advances-2014} that if $K$ is a centrally symmetric convex body in ${\mathbb R}^n$ and if $g$ is
even and continuous on $K$ then
\begin{equation*}\mu (K)\ls \gamma_{n,1}\frac{n}{n-1}\sqrt{n}\,|K|^{\frac{1}{n}}\,\max_{\xi\in S^{n-1}}\mu (K\cap \xi^{\perp }),\end{equation*}
where, more generally, $\gamma_{n,k}=|B_2^n|^{\frac{n-k}{n}}/|B_2^{n-k}|<1$ for all $1\ls k\ls n-1$.
In \cite{Koldobsky-GAFA-2014}, Koldobsky obtained estimates for the lower dimensional sections: if $K$ is a centrally symmetric convex body in
${\mathbb R}^n$ and $g$ is even and continuous on $K$ then
\begin{equation*}\mu (K)\ls \gamma_{n,k}\frac{n}{n-k}(\sqrt{n})^k\,|K|^{\frac{k}{n}}\,\max_{H\in G_{n,n-k}}\mu (K\cap H)\end{equation*}
for every $1\ls k\ls n-1$. A different proof of this fact was given in \cite{CGL}: the method in this work allows one to drop the symmetry and continuity assumptions: Let $K$ be a convex body in ${\mathbb R}^n$ with $0\in {\rm int}(K)$. Let $g$ be a bounded
non-negative measurable function on ${\mathbb R}^n$ and let $\mu $ be the measure on ${\mathbb R}^n$ with density $g$. For every $1\ls k\ls n-1$,
\begin{equation*}\mu (K)\ls \left (c_4\sqrt{n-k}\right )^k\,|K|^{\frac{k}{n}}\,\max_{H\in G_{n,n-k}}\mu (K\cap H),\end{equation*}
In fact, the proof leads to the stronger estimate
\begin{equation*}\mu (K)\ls \left (c_5\sqrt{n-k}\right )^k\,|K|^{\frac{k}{n}}\left (\int_{G_{n,n-k}}\mu (K\cap H)^n\,d\nu_{n,n-k}(H)\right )^{\frac{1}{n}} .\end{equation*}
In this work we study the slicing problem for the surface area and other quermassintegrals of convex bodies.
In Section~3 we recall related results regarding the surface area of projections of convex bodies.
However, the natural generalization of the slicing problem for sections that we stated
in the beginning of this introduction has not been studied. As far as we know there are no general inequalities comparing
the surface area $S(K)$ of a convex body $K$ in ${\mathbb R}^n$ to the average or maximal surface area of its
hyperplane or lower dimensional sections.

\smallskip

{\it Main results.} Our first main result states that it is not possible to have an inequality such as \eqref{eq:slicing-surface-1}.

\begin{theorem}\label{th:negative-1}For any $n\gr 2$ one has that
\begin{equation*}\sup\Big\{\frac{S(K)}{|K|^{\frac{1}{n}}\max\limits_{\xi\in S^{n-1}}S(K\cap\xi^{\perp })}:K\;\mathrm{is \;a \;
centrally\;symmetric\;convex\; body\; in}\;{\mathbb R}^n\Big\}=+\infty .\end{equation*}
\end{theorem}

For the proof of Theorem~\ref{th:negative-1} we show that for any $\alpha >0$ one may construct
a centrally symmetric ellipsoid ${\cal E}$ such that
$$S({\cal E})>\alpha |{\cal E}|^{\frac{1}{n}}\max_{\xi\in S^{n-1}}S({\cal E}\cap\xi^{\perp }).$$
In order to do this, for a given ellipsoid ${\cal E}$ in ${\mathbb R}^n$ we need to know the $(n-1)$-dimensional
section of ${\cal E}$ that has the largest surface area. This is a natural question of independent interest, which
we answer in Section~4. We show that if ${\cal E}$ is an origin symmetric ellipsoid in ${\mathbb R}^n$, and if
$a_1\ls a_2\ls\cdots \ls a_n$ are the lengths and $e_1,e_2,\ldots ,e_n$ are the corresponding directions of its semi-axes, then
$$S({\cal E}\cap \xi^{\perp })\ls S({\cal E}\cap e_1^{\perp })$$
for every $\xi\in S^{n-1}$. Then, we combine this information with a formula of Rivin \cite{Rivin-2007}
for the surface area of an ellipsoid: If ${\cal E}$ is an ellipsoid in
${\mathbb R}^n$ with semi-axes $a_1\ls a_2\ls\cdots \ls a_n$ in the directions of $e_1,\ldots ,e_n$ then
\begin{equation}\label{eq:rivin}S({\cal E})=n\,|{\cal E}|\,\int_{S^{n-1}}\Big (\sum_{i=1}^n\frac{\xi_i^2}{a_i^2}\Big)^{1/2}d\sigma (\xi ).\end{equation}

In fact, as we will see, for any $k-$dimensional subspace $H$ and any $0\ls j\ls k-1$ we have that
$$W_j({\cal E}\cap F_k)\ls W_j({\cal E}\cap H)\ls W_j({\cal E}\cap E_k)$$
and
$$W_j(P_{F_k}({\cal E}))\ls W_j(P_H({\cal E}))\ls W_j(P_{E_k}({\cal E})),$$
where $F_k=\mathrm{span}\{e_1,\ldots ,e_k\}$, $E_k=\mathrm{span}\{e_{n-k+1},\ldots, e_n\}$ and $W_j$ denotes the $j$-th quermassintegral of a convex body
(see Section~2 for the necessary definitions). These results are the analogues of a known fact for the maximal and minimal
volume of $k$-dimensional sections and projections of ellipsoids (see Section~4 for further details and references).
As a consequence we obtain a more general negative result about all the quermassintegrals of sections and projections of convex bodies.

\begin{theorem}\label{th:negative-2}Let $n\gr 3$, $1\ls k\ls n-2$ and $1\ls j\ls n-k-1$. Then,
$$\sup\left\{\frac{W_j(K)}{|K|^{\frac{k}{n}}\,\max\limits_{H\in G_{n,n-k}}W_j(K\cap H)}:K\;\hbox{is a convex body in}\;{\mathbb R}^n\right\}=+\infty .$$
In fact, we also have that
$$\sup\left\{\frac{W_j(K)}{|K|^{\frac{k}{n}}\,\max\limits_{H\in G_{n,n-k}}W_j(P_H(K))}:K\;\hbox{is a convex body in}\;{\mathbb R}^n\right\}=+\infty .$$
\end{theorem}

In Section~5 we provide some estimates, in the positive direction, for the surface area version of the slicing problem.
However, they depend on the parameter
$$t(K)=\left(\frac{|K|}{|r(K)B^n_2|}\right)^{\frac{1}{n}}$$
where $r(K)$ is the inradius of $K$, i.e. the largest value of $r>0$ for which there exists $x_0\in K$ such that $x_0+rB_2^n\subseteq K$.
More precisely, we show:

\begin{theorem}\label{th:positive-1}
Let $K$ be a convex body in $\mathbb{R}^n$ with barycenter at $0$. Then,
$$S(K)\ls \frac{d_n}{d_{n-k}}(c_1L_K)^{\frac{k(n-k-1)}{n-k}}t(K)|K|^{\frac{k}{n}}\max_{H\in G_{n,n-k}}S(K\cap H)$$
where $d_s=s\omega_s^{1/s}$ and $c_1>0$ is an absolute constant.
\end{theorem}

Note that in the case $k=1$ (the hyperplane case) we have $\frac{d_n}{d_{n-1}}(c_1L_K)^{\frac{n-2}{n-1}}\approx L_K$. We also provide a variant
of Theorem~\ref{th:positive-1} for the ratio $\frac{S(K)}{|K|}$.

\begin{theorem}\label{th:positive-2}
Let $K$ be a convex body in $\mathbb{R}^n$ with $0\in {\rm int}(K)$. Then, for all $1\ls k\ls n-1$ we have that
$$\frac{S(K)}{|K|}\ls\frac{n}{n-k}t(K)\max_{H\in G_{n,n-k}}\frac{S(K\cap H)}{|K\cap H|}.$$
\end{theorem}

Our proof of these results involves the Grinberg/Busemann-Straus inequality, an estimate for the dual affine quermassintegrals of a
convex body, and the classical Aleksandrov inequalities. In fact, the same more or less argument leads to similar
results for any quermassintegral and not only for surface area (the precise statements are given in Section~5).

Our starting point in Section~6 are two simple inequalities relating the surface area of a convex body $K$ to its volume.
One has
\begin{equation*}r(K)S(K)\ls n|K|\ls R(K)S(K),\end{equation*}
where $r(K)$ and $R(K)$ denote the inradius and the circumradius of $K$ respectively. In the case of an ellipsoid, we observe that
\eqref{eq:rivin}, the formula which is used for the proof of Theorem~\ref{th:negative-1}, can be rewritten as
$$S({\cal E})=n|{\cal E}|M_2({\cal E}),$$
where $M_2^2({\cal E})=\int_{S^{n-1}}\|\xi\|_{{\cal E}}^2d\sigma (\xi )$ (and $\|\xi\|_K$ denotes the Minkowski
functional of a convex body $K$ with $0\in {\rm int}(K)$). Using the fact that
$M_2({\cal E})\approx M({\cal E})=\int_{S^{n-1}}\|\xi\|_{{\cal E}}d\sigma (\xi )$, we get
$$S({\cal E})\approx n|{\cal E}|M({\cal E}).$$
Note that $r(K)\ls M(K)^{-1}\ls w(K)\ls R(K)$, where $w(K)$ is the mean width of $K$.
Therefore, for a convex body $K\subset\mathbb{R}^n$, we naturally introduce the parameters
$$p(K)=\frac{S(K)}{|K|M(K)}\quad\hbox{and}\quad q(K)=\frac{w(K)S(K)}{|K|}$$ and we ask for upper and lower bounds for them.
Theorems \ref{th:2.1} and \ref{parameter_lower} show that there are absolute constants $c_1, c_2>0$ such that
for every convex body $K\in\mathbb{R}^n$ we have
$$c_1\sqrt{n}\ls p(K)\ls c_2n^{3/2}.$$
Moreover, the order of $n$ cannot be improved in both the upper and the lower bound. However, we prove that if $K$ is in
a classical position, such as John's position or the minimal surface area position or the isotropic position, then
these estimates can be improved. The situation is different with $q(K)$. We show that $q(K)\gr n$ for every
convex body $K$ in ${\mathbb R}^n$, while in general there can be no upper bound in any fixed dimension: for any
$n\gr 2$ one has $\sup\{q(K):K\;\hbox{is a convex body in}\;{\mathbb R}^n\}=+\infty $.

In Section~7 we study a variant of our main problem. Our starting point is a surface area variant of the equivalence of the
isomorphic Busemann--Petty problem with the slicing problem:
Assuming that there is a constant $\gamma_n$ such that
if $K$ and $D$ are centrally symmetric convex bodies in $\mathbb{R}^n$ that satisfy
$$S(K\cap\xi^{\perp})\ls S(D\cap\xi^{\perp})$$ for all $\xi\in S^{n-1}$, then $S(K)\ls \gamma_n S(D)$,
one can see that there is some constant $c(n)$ such that
\begin{equation}\label{slicing-15}
S(K)\ls c(n)S(K)^{\frac{1}{n-1}}\max_{\xi\in S^{n-1}} S(K\cap\xi^{\perp})
\end{equation}
for every convex body $K$ in ${\mathbb R}^n$. We show that an inequality of this type holds true in general.

\begin{theorem}\label{sec-7}Let $K$ be a convex body in ${\mathbb R}^n$. Then,
$$S(K)\ls A_nS(K)^{\frac{1}{n-1}}\max\limits_{\xi\in S^{n-1}}S(K\cap\xi^{\perp })$$
where $A_n>0$ is a constant depending only on $n$.
\end{theorem}

We obtain this result for an arbitrary ellipsoid; then, it is not hard to extend it to any convex body, using John's theorem.
The value of the constant $A_n$ that one can obtain in this way is clearly not optimal and it would be interesting to determine its
best possible dependence on the dimension $n$.

\section{Notation and background information}

We work in ${\mathbb R}^n$, which is equipped with the standard inner product $\langle\cdot ,\cdot\rangle $. We denote by $\|\cdot \|_2$
the Euclidean norm, and write $B_2^n$ for the Euclidean unit ball and $S^{n-1}$ for the unit sphere.
Volume is denoted by $|\cdot |$. We write $\omega_n$ for the volume of $B_2^n$ and $\sigma $ for the rotationally invariant probability
measure on $S^{n-1}$. The Grassmann manifold $G_{n,k}$ of all $k$-dimensional subspaces of ${\mathbb R}^n$ is equipped with the Haar probability
measure $\nu_{n,k}$. For every $1\ls k\ls n-1$ and $H\in G_{n,k}$ we write $P_H$ for the orthogonal projection from $\mathbb R^{n}$ onto $H$.

The letters $c,c^{\prime }, c_1, c_2$ etc. denote absolute positive constants which may change from line to line. Whenever we write
$a\approx b$, we mean that there exist absolute constants $c_1,c_2>0$ such that $c_1a\ls b\ls c_2a$.  Also, if $K,D\subseteq \mathbb R^n$
we will write $K\approx D$ if there exist absolute constants $c_1, c_2>0$ such that $ c_{1}K\subseteq D \subseteq c_{2}K$.

\smallskip

A convex body in ${\mathbb R}^n$ is a compact convex subset $K$ of ${\mathbb R}^n$ with non-empty interior. We say that $K$ is
centrally symmetric if $x\in K$ implies that $-x\in K$, and that $K$ is centered if its barycenter $\frac{1}{|K|}\int_Kx\,dx $ is at the origin.
The support function of a convex body $K$ is defined by $h_K(y)=\max \{\langle x,y\rangle :x\in K\}$, and the mean width of $K$ is
\begin{equation*}w(K)=\int_{S^{n-1}}h_K(\xi )\,d\sigma (\xi ). \end{equation*}
The circumradius of $K$ is the quantity $R(K)=\max\{ \|x\|_2:x\in K\}$ i.e. the smallest $R>0$ for which $K\subseteq RB_2^n$.
We write $r(K)$ for the inradius of $K$, the largest $r>0$
for which there exists $x_0\in K$ such that $x_0+rB_2^n\subseteq K$. If $0\in {\rm int}(K)$ then we define the polar body $K^{\circ }$ of $K$ by
\begin{equation*}K^{\circ }:=\{ y\in {\mathbb R}^n: \langle x,y\rangle \ls 1 \;\hbox{for all}\; x\in K\}. \end{equation*}
The volume radius of $K$ is the quantity ${\rm vrad}(K)=\left (|K|/|B_2^n|\right )^{1/n}$.
Integration in polar coordinates shows that if the origin is an interior point of $K$ then the volume radius of $K$ can be expressed as
\begin{equation*}{\rm vrad}(K)=\left (\int_{S^{n-1}}\|\xi\|_K^{-n}\,d\sigma (\xi )\right)^{1/n},\end{equation*}
where $\|x\|_K=\min\{ t\gr 0:x\in tK\}$ is the Minkowski functional of $K$. We also define
\begin{equation*}M(K)=\int_{S^{n-1}}\|\xi\|_K\,d\sigma (\xi ).\end{equation*}
A convex body $K$ in ${\mathbb R}^n$ is called isotropic if it has volume $1$, it is centered and its inertia matrix is a multiple of the identity matrix:
there exists a constant $L_K >0$ such that
\begin{equation*}\int_K\langle x,\xi\rangle^2dx =L_K^2\end{equation*}
for all $\xi\in S^{n-1}$. The constant $L_K$ is the isotropic constant of $K$.

From Minkowski's fundamental theorem we know that if $K_1,\ldots ,K_m$ are non-empty, compact convex
subsets of ${\mathbb R}^n$, then the volume of $t_1K_1+\cdots +t_mK_m$ is a homogeneous polynomial of degree $n$ in
$t_i>0$. That is,
\begin{equation*}|t_1K_1+\cdots +t_mK_m|=\sum_{1\ls i_1,\ldots ,i_n\ls m}
V(K_{i_1},\ldots ,K_{i_n})t_{i_1}\cdots t_{i_n},\end{equation*}
where the coefficients $V(K_{i_1},\ldots ,K_{i_n})$ are chosen to be invariant under permutations of their arguments. The coefficient $V(K_1,\ldots ,K_n)$ is the mixed volume of $K_1,\ldots ,K_n$. In particular, if $K$ and $D$ are two convex bodies in ${\mathbb R}^n$
then the function $|K+tD|$ is a polynomial in $t\in [0,\infty )$:
\begin{equation*}|K+tD|=\sum_{j=0}^n \binom{n}{j} V_{n-j}(K,D)\;t^j,\end{equation*}
where $V_{n-j}(K,D)= V((K,n-j),(D,j))$ is the $j$-th mixed volume of $K$ and $D$ (we use  the notation $(D,j)$ for $D,\ldots ,D$ $j$-times).
If $D=B_2^n$ then we set $W_j(K):=V_{n-j}(K,B_2^n)=V((K, n-j), (B_2^n, j))$; this is the $j$-th quermassintegral of $K$.
Note that
\begin{equation*}V_{n-1}(K,D)={\frac{1}{n}} \lim_{t\to 0^+}{\frac{|K+tD|-|K|}{t}},\end{equation*}
and by the Brunn-Minkowski inequality we see that
\begin{equation*}V_{n-1}(K,D)\gr |K|^{\frac{n-1}{n}}|D|^{\frac{1}{n}}\end{equation*}
for all $K$ and $D$ (this is Minkowski's first inequality). The mixed volume $V_{n-1}(K,D)$ can be expressed as
\begin{equation}\label{eq:not-1}V_{n-1}(K,D)={\frac{1}{n}}\int_{S^{n-1}}h_D(\theta )d\sigma_K(\theta ),\end{equation}
where $\sigma_K$ is the surface area measure of $K$; this is the Borel measure
on $S^{n-1}$ defined by
\begin{equation*}\sigma_K(A)=\lambda (\{x\in {\rm bd}(K):\;{\rm the}\;{\rm outer}
\;{\rm normal}\;{\rm to}\;K\;{\rm at}\;x\;{\rm belongs}\;
{\rm to}\;A\}),\end{equation*}
where $\lambda $ is the Hausdorff measure on ${\rm bd}(K)$. In particular, the surface area $S(K):=\sigma_K(S^{n-1})$ of $K$ satisfies
\begin{equation*}S(K)=nW_1(K).\end{equation*}
Kubota's integral formula expresses the quermassintegral $W_j(K)$ as an average of the volumes of
$(n-j)$-dimensional projections of $K$:
\begin{equation*}W_j(K)=\frac{\omega_n}{\omega_{n-j}}\int_{G_{n,n-j}}
|P_H(K)|d\nu_{n,n-j}(H).\end{equation*}Applying this formula for
$j=n-1$ we see that \begin{equation*}W_{n-1}(K)=\omega_n
w(K).\end{equation*}
It is convenient to work with a normalized variant of $W_{n-j}(K)$. If we set
\begin{equation}\label{eq:aleksandrov-1}Q_k(K)=\left
(\frac{W_{n-k}(K)}{\omega_n}\right )^{\frac{1}{k}}=\left
(\frac{1}{\omega_k}\int_{G_{n,k}}|P_H(K)|\,d\nu_{n,k}(H)\right )^{\frac{1}{k}},\end{equation}
then $k\mapsto Q_k(K)$ is decreasing. This is a consequence of the Aleksandrov-Fenchel inequality (see \cite{Burago-Zalgaller-book} and \cite{Schneider-book}). In particular, for every $1\ls k\ls n-1$ we have
\begin{equation}\label{eq:aleksandrov-2}{\rm vrad}(K)=\left (\frac{|K|}{\omega_n}\right )^{\frac{1}{n}}\ls \left (\frac{1}{\omega_k}\int_{G_{n,k}}|P_H(K)|\,d\nu_{n,k}(H)\right )^{\frac{1}{k}}\ls w(K).\end{equation}
We will also use some estimates for the (normalized) dual affine quermassintegrals. For every convex body $K$
in ${\mathbb R}^n$ and every $1\ls k\ls n-1$ we consider the quantity
$$ \widetilde{\Phi}_{[k]}(K):=\frac{1}{|K|^\frac{n-k}{nk}}\left(\int_{G_{n,k}}|K\cap H^{\perp}|^nd\nu_{n,k}\right)^\frac{1}{kn}.$$
It was proved independently by Busemann and Straus \cite{Busemann-Straus-1960}, and  Grinberg \cite{Grinberg-1990}
that $\tilde{\Phi}_{[k]}(K)\ls\tilde{\Phi}_{[k]}(B_2^n)\ls c_1$,
where $c_1>0$ is an absolute constant. Dafnis and Paouris showed in \cite{Dafnis-Paouris-2012} that if $K$ is a centered convex body
in ${\mathbb R}^n$ then $$\tilde{\Phi}_{[k]}(K)\gr \frac{c_2}{L_K},$$
where $c_2>0$ is an absolute constant and $L_K$ is the isotropic constant of $K$. In particular, assuming that $L_K\ls C$ for
an absolute constant we have that
$\tilde{\Phi}_{[k]}(K)\approx 1$ for every centered convex body $K$ in ${\mathbb R}^n$ and all $1\ls k\ls n-1$.

\smallskip

We refer to the books \cite{Gardner-book} and \cite{Schneider-book} for basic facts from the Brunn-Minkowski theory and to the book
\cite{AGA-book} for basic facts from asymptotic convex geometry. We also refer to \cite{BGVV-book} for more information on isotropic convex bodies.

\section{Surface area of projections}

Related to our work is the article \cite{GKV} of Giannopoulos, Koldobsky and Valettas, which provides general inequalities that compare the surface area $S(K)$ of a convex body $K$ in ${\mathbb R}^n$ to the minimal, average or maximal surface area of its hyperplane or lower dimensional projections. The same questions are also discussed for all the quermassintegrals. Starting from two inequalities of Koldobsky about the surface area of hyperplane projections of projection bodies (see \cite{Koldobsky-2013} and \cite{Koldobsky-2015}) the authors in \cite{GKV}
obtain inequalities for the surface area of hyperplane projections of an arbitrary convex body $K$ in ${\mathbb R}^n$.
Let $\partial_K$ denote the minimal surface area parameter of $K$, defined by
\begin{equation*}\partial_K:=\min\Big\{ S(T(K))/|T(K)|^{\frac{n-1}{n}}:T\in GL(n)\Big\}.\end{equation*}
By the isoperimetric and the reverse isoperimetric inequality (see \cite[Chapter~2]{AGA-book}) it is known that $c_1\sqrt{n}\ls \partial_K\ls c_2n$
for every convex body $K$ in ${\mathbb R}^n$, where $c_1,c_2>0$
are absolute constants, It is proved in \cite{GKV} that there exists an absolute constant $c_3>0$ such that, for every convex body $K$ in ${\mathbb R}^n$,
\begin{equation*}|K|^{\frac{1}{n}}\,\min_{\xi\in S^{n-1}}S(P_{\xi^{\perp }}(K))\ls \frac{2b_n\partial_K}{n\omega_n^{\frac{1}{n}}}\,S(K)\ls\frac{c_3\partial_K}{\sqrt{n}}\,S(K),\end{equation*}
where $b_n=\frac{(n-1)\omega_{n-1}}{n\omega_n^{\frac{n-1}{n}}}\approx 1$. This inequality is sharp e.g. for the Euclidean unit ball. Since  $c_3\partial_K/\sqrt{n}\ls c\sqrt{n}$ for every convex body $K$ in ${\mathbb R}^n$, one has the general upper bound
\begin{equation*}|K|^{\frac{1}{n}}\,\min_{\xi\in S^{n-1}}S(P_{\xi^{\perp }}(K))\ls c_4\sqrt{n}\,S(K).\end{equation*}
In the opposite direction, it is proved in \cite{GKV} that if $K$ is a convex body in ${\mathbb R}^n$ then
\begin{equation*}\int_{S^{n-1}}S(P_{\xi^{\perp }}(K))\,d\sigma (\xi )\gr c_5\,S(K)^{\frac{n-2}{n-1}},\end{equation*}
where $c_5>0$ is an absolute constant. A consequence of this inequality is that if $K$ is in the minimal surface area, minimal mean width, isotropic, John or L\"{o}wner position (see \cite[Chapter~2]{AGA-book}) then
\begin{equation*}|K|^{\frac{1}{n}}\,\int_{S^{n-1}}S(P_{\xi^{\perp }}(K))\,d\sigma (\xi )\gr c_6\,S(K),\end{equation*}
where $c_6>0$ is an absolute constant. In particular,
\begin{equation*}|K|^{\frac{1}{n}}\,\max_{\xi\in S^{n-1}}S(P_{\xi^{\perp }}(K))\gr c_6\,S(K).\end{equation*}
In fact, these inequalities continue to hold as long as
\begin{equation*}S(K)^{\frac{1}{n-1}}\ls c_7|K|^{\frac{1}{n}}\end{equation*}
for an absolute constant $c_7>0$. This is a mild condition which is satisfied not only by the classical positions but also
by all reasonable positions of $K$. It should be noted that the question whether there exists a constant $\alpha_n$ such that
\begin{equation*}S(K)\ls\alpha_n\,|K|^{\frac{1}{n}}\max_{\xi\in S^{n-1}}S(P_{\xi^{\perp }}(K))\end{equation*}
for all convex bodies $K$ in ${\mathbb R}^n$ is left open in \cite{GKV}. As we will see in the next section, it has
a negative answer. In \cite{GKV} the same questions are studied for the quermassintegrals $V_{n-k}(K)=V((K, n-k), (B_2^n,k))$ of a convex body $K$ and the
corresponding quermassintegrals of its hyperplane projections.

\section{Ellipsoids and a negative answer to the problem}

In this section we provide a negative answer to the slicing problem for the surface area.

\begin{theorem}\label{th:negative-1b}For any $n\gr 2$ and any $\alpha >0$ there exists a centrally symmetric convex body $K$
in ${\mathbb R}^n$ such that
\begin{equation*}S(K)>\alpha\,|K|^{\frac{1}{n}}\max_{\xi\in S^{n-1}}S(P_{\xi^{\perp }}(K))\gr \alpha\,|K|^{\frac{1}{n}}\max_{\xi\in S^{n-1}}S(K\cap\xi^{\perp }).\end{equation*}
\end{theorem}

In fact, our examples will be given by ellipsoids. They will be based on the next result which answers a natural question
and might be useful in other situations too.

\begin{theorem}\label{th:max-surface-ellipsoid-section}Let ${\cal E}$ be an origin symmetric ellipsoid in ${\mathbb R}^n$ and write
$a_1\ls a_2\ls\cdots \ls a_n$ for the lengths and $e_1,e_2,\ldots ,e_n$ for the corresponding directions of its semi-axes.
If $1\ls k\ls n-1$ then for any $H\in G_{n,k}$ and any $0\ls j<k$ we have that
$$W_j({\cal E}\cap F_k)\ls W_j({\cal E}\cap H)\ls W_j(P_H({\cal E}))\ls W_j({\cal E}\cap E_k),$$
where $F_k=\mathrm{span}\{e_1,\ldots, e_k\}$ and $E_k=\mathrm{span}\{e_{n-k+1},\ldots, e_n\}.$
In particular, for every $\xi\in S^{n-1}$,
$$S({\cal E}\cap \xi^{\perp })\ls S(P_{\xi^{\perp }}({\cal E}))\ls S({\cal E}\cap e_1^{\perp }).$$
\end{theorem}

The analogue of Theorem~\ref{th:max-surface-ellipsoid-section} for the volume of sections and projections
of ellipsoids is known to be true (for a proof see \cite{Klartag-VMilman-2005b} and
\cite{Dafnis-Paouris-2010}). With the same notation, for all $1\ls k\ls n-1$ one has
\begin{equation*}
\min_{H \in G_{n,k}} |{\cal E} \cap H| =\min_{H \in G_{n,k}}|P_H({\cal E})| =\omega_k  \prod_{i=1}^ka_i
\end{equation*}
and
\begin{equation*}
\max_{H \in G_{n,k}} |{\cal E} \cap H| = \max_{H \in G_{n,k}}|P_H({\cal E})|= \omega_k \prod_{i=n-k+1}^na_i.
\end{equation*}

For the proof of Theorem~\ref{th:max-surface-ellipsoid-section} we will use the following form of Cauchy's interlacing theorem
(see \cite[pp.~64]{Rao-book}).

\begin{theorem}\label{thm:interlace}
Let $A$ be a symmetric $n\times n$ matrix and consider the $k\times k$  matrix $B=PAP^{*}$, where $k\ls n$ and $P$ is the orthogonal projection
onto a subspace of dimension $k$. If the eigenvalues of $A$ are $\lambda_1\ls \lambda_2\ls\cdots\ls \lambda_n$, and those of $B$ are $\mu_1\ls \mu_2\ls\cdots\ls \mu_k$,
then for all $i\ls k$ we have
$$\lambda_{i}\ls \mu_{i}\ls \lambda_{n-k+i}.$$
\end{theorem}

Now, let ${\cal E}$ be an origin symmetric ellipsoid in ${\mathbb R}^n$ and let $a_1\ls\cdots\ls a_n$ be the lengths of its principal semi-axes. We can write ${\cal E}=\{x\in {\mathbb R}^n:\langle Ax,x\rangle\ls 1\}$, where $A$ is an $n\times n$ symmetric positive definite matrix.
The relation between the eigenvalues $\lambda_1(A)\ls \cdots \ls \lambda_n(A)$ of $A$ and the lengths of the principal semi-axes of the ellipsoid ${\cal E}$ is given by  $$a_j=\frac{1}{\sqrt{\lambda_{n-j+1}(A)}}.$$
A $k$-dimensional section of ${\cal E}$ can be obtained by restriction onto a $k-$dimensional subspace $H$:
$$\mathcal{E}\cap H=\{x\in H\colon\ \langle Ax, x\rangle\ls 1\}.$$
Let $b_1\ls\cdots\ls b_k$ be the lengths of the principal semi-axes of ${\cal E}\cap H$.
If $\{u_1, \ldots,u_k\}$ is an orthonormal basis of $H$ then we can write any $x\in H$ as
$$x = y_1u_1 +\cdots+ y_ku_k,$$ for some vector
$y = (y_1,\ldots, y_k)\in\mathbb{R}^k$. Thus, we can write $x=Uy$, where $U$ is an $n\times k$ matrix with columns $u_i$ for $1\ls i\ls k$. Using this language we can write
$$\mathcal{E}\cap H=\{y\in\mathbb{R}^k\colon\ \langle AUy, Uy\rangle\ls 1\}=\{y\in\mathbb{R}^k\colon\ \langle U^{*}AUy, y\rangle\ls 1\}.$$
By our previous observations we conclude that  the $j$-th principal semi-axis of $\mathcal{E}\cap H$ is given by
$$b_j=(\lambda_{k-j+1}(U^{*}AU))^{-1/2}.$$
We can now use Theorem \ref{thm:interlace} for $i=k-j+1$ to get
$$\lambda_{k-j+1}(A)\ls \lambda_{k-j+1}(U^{*}AU)\ls\lambda_{n-j+1}(A),$$
which implies that
$$a_j\ls b_j\ls a_{n-k+j}.$$
Therefore we obtain the following geometric consequence of Theorem~\ref{thm:interlace}.

\begin{lemma}[Generalisation of Rayleigh's formula]\label{lem:rayl}Let ${\cal E}$ be an origin symmetric ellipsoid in ${\mathbb R}^n$ and write
$a_1\ls a_2\ls\cdots \ls a_n$ for the lengths of its semi-axes. If $H$ is a $k$-dimensional subspace of ${\mathbb R}^n$
then ${\cal E}\cap H$ is an origin symmetric ellipsoid and its semi-axes $b_1\ls b_2\ls\cdots\ls b_{k}$ satisfy
$$a_j\ls b_j\ls a_{n-k+j},$$
for all $1\ls j\ls k$.
\end{lemma}

\begin{proof}[Proof of Theorem~$\ref{th:max-surface-ellipsoid-section}$]We may assume that $\{e_1,\ldots ,e_n\}$ is the standard
basis of ${\mathbb R}^n$ and write
$${\cal E}=\Big\{x\in {\mathbb R}^n:\sum_{j=1}^n\frac{x_j^2}{a_j^2}\ls 1\Big\}.$$
Consider the ellipsoid
$${\cal E}^{\prime }:=\Big\{x\in {\mathbb R}^n:\sum_{j=1}^{n-k+1}\frac{x_j^2}{a_{n-k+1}^2}+\sum_{j=n-k+2}^{n}\frac{x_j^2}{a_{j}^2}\ls 1\Big\}.$$
We clearly have ${\cal E}\subseteq {\cal E}^{\prime }$. Then, for any $k-$dimensional subspace $H$
we have that ${\cal E}\cap H\subseteq {\cal E}^{\prime }\cap H$, and hence Kubota's formula implies that
$$W_j({\cal E}\cap H)\ls W_j({\cal E}^{\prime }\cap H)$$
for all $0\ls j< k$. On the other hand, if $b_1\ls b_2\ls \cdots \ls b_{k}$ are the lengths of the semi-axes of the
ellipsoid ${\cal E}^{\prime }\cap H$ then Lemma~\ref{lem:rayl} shows that
$$a_{n-k+1}\ls b_{1}\ls a_{n-k+1}\ls b_{2}\ls a_{n-k+2}\ls\cdots\ls b_k\ls a_n,$$
therefore $a_{n-k+1}=b_{1}$ and $b_j\ls a_{n-k+j}$ for all $1\ls j\ls k$. Thus, all the semi-axes of ${\cal E}^{\prime }\cap H$
are smaller than or equal to the corresponding ones of ${\cal E}^{\prime }\cap E_k$, which implies that
$$W_j({\cal E}^{\prime }\cap H)\ls W_j({\cal E}^{\prime }\cap E_k)$$
for all $1\ls j\ls k$. Combining the above we get
$$W_j({\cal E}\cap H)\ls W_j({\cal E}^{\prime }\cap H)\ls W_j({\cal E}^{\prime }\cap E_k)
=W_j({\cal E}\cap E_k),$$
where the last equality follows from the observation that ${\cal E}^{\prime }\cap E_k={\cal E}\cap E_k$. The proof of the
inequality $W_j({\cal E}\cap F_k)\ls W_j({\cal E}\cap H)$ is similar.

For the proof of the corresponding result for projections we may use a duality argument.
Given two ellipsoids $\mathcal{E}_1$ and $\mathcal{E}_2$, with semi-axes $a_1\ls\cdots\ls a_n$ and $b_1\ls\cdots\ls b_n$ respectively,
we will write $\mathcal{E}_1\preceq\mathcal{E}_2$ if $a_i\ls b_i$ for all $i$. Using this notation, what we have proved is that
$$\mathcal{E}\cap F_k\preceq\mathcal{E}\cap H\preceq\mathcal{E}\cap E_k$$
for every $H\in G_{n,k}$. Now, we start with the ellipsoid $\mathcal{E}^{\circ }$. Since the lengths of the semi-axes of $\mathcal{E}^{\circ}$
are the reciprocals of the ones of $\mathcal{E}$, we see that
$$\mathcal{E}^{\circ}\cap E_k\preceq \mathcal{E}^{\circ}\cap H\preceq \mathcal{E}^{\circ}\cap F_k,$$
and hence their corresponding polars satisfy
$$P_{F_k}(\mathcal{E})\preceq P_{H}(\mathcal{E})\preceq P_{E_k}(\mathcal{E}).$$
The result follows from these observations.
\end{proof}

\begin{remark}\rm A formula which is related to this discussion is proved in \cite{Kab-Zap-2016}. If ${\cal F}$ is an ellipsoid
in ${\mathbb R}^k$ then
\begin{equation}\label{eq:KZ}W_{k-j}({\cal F})=\frac{|{\cal F}|}{\omega_k}\,W_j({\cal F}^{\circ })\end{equation}
for every $1\ls j\ls k-1$. Let ${\cal E}$ be an ellipsoid in ${\mathbb R}^n$, and let $1\ls k\ls n-1$ and $H\in G_{n,k}$.
Keeping the notation $E_k$ and $F_k$ as above, and applying \eqref{eq:KZ} for the ellipsoid ${\cal E}\cap H$, for every $1\ls j\ls k-1$ we see that
$$\frac{W_{k-j}({\cal E}\cap H)}{|{\cal E}\cap H|}=\frac{1}{\omega_k}\,|W_j(P_H({\cal E}^{\circ }))|
\ls \frac{1}{\omega_k}\,|W_j(P_{F_k}({\cal E}^{\circ }))|=\frac{W_{k-j}({\cal E}\cap F_k)}{|{\cal E}\cap F_k|}.$$
In other words, the ratio $W_{k-j}({\cal E}\cap H)/|{\cal E}\cap H|$ is maximized when $H=F_k$, and similarly it is minimized
when $H=E_k$. Analogously, applying \eqref{eq:KZ} for the ellipsoid $P_H({\cal E})$, for every $1\ls j\ls k-1$ we see that
$$\frac{W_{k-j}(P_H({\cal E}))}{|P_H({\cal E})|}=\frac{1}{\omega_k}\,|W_j({\cal E}^{\circ }\cap H)|
\ls \frac{1}{\omega_k}\,|W_j({\cal E}^{\circ }\cap F_k)|=\frac{W_{k-j}(P_{F_k}{\cal E})}{|P_{F_k}({\cal E})|}.$$
In other words, the ratio $W_{k-j}(P_H({\cal E}))/|P_H({\cal E})|$ is also maximized when $H=F_k$, and similarly it is minimized
when $H=E_k$.
\end{remark}

We pass now to the proof of Theorem~\ref{th:negative-1b} and of the more general Theorem~\ref{th:negative-general}

\smallskip

\begin{proof}[Proof of Theorem~$\ref{th:negative-1b}$]We shall use the next formula of Rivin (see \cite{Rivin-2007}):
If ${\cal E}$ is an ellipsoid in ${\mathbb R}^n$ with semi-axes $a_1\ls\cdots \ls a_n$ in the directions of $e_1,\ldots ,e_n$ then
$$S({\cal E})=n\,|{\cal E}|\,\int_{S^{n-1}}\Big (\sum_{i=1}^n\frac{\xi_i^2}{a_i^2}\Big)^{1/2}d\sigma (\xi ).$$
Recall also that for any norm $\|\cdot\|$ on ${\mathbb R}^n$ we have that
$${\mathbb E}\|G\|=d_n\int_{S^{n-1}}\|\xi\|\,d\sigma (\xi ),$$
where $G$ is a standard Gaussian random vector and $d_n\sim\sqrt{n}$.

Now assume that there exists a constant $\alpha_n>0$ such that we have the following inequality for ellipsoids:
\begin{equation}\label{eq:max}S({\cal E})\ls \alpha_n|{\cal E}|^{1/n}\max_{\xi\in S^{n-1}}S({\cal E}\cap\xi^{\perp }).\end{equation}
From Theorem~\ref{th:max-surface-ellipsoid-section} we know that the maximum is attained for the section ${\cal E}\cap e_1^{\perp}$. Then we have
$$\max_{\xi\in S^{n-1}}S({\cal E}\cap\xi^{\perp })=S({\cal E}\cap e_1^{\perp })=(n-1)\,|{\cal E}\cap e_1^{\perp }|\,\int_{S^{n-2}}\Big (\sum_{i=2}^n\frac{\xi_i^2}{a_i^2}\Big)^{1/2}d\sigma (\xi ).$$
We may assume that $\prod_{i=1}^na_i=1$. Then, we can rewrite \eqref{eq:max} as
\begin{equation*}n\omega_n\cdot\frac{1}{d_n}{\mathbb E}\Big[\Big (\sum_{i=1}^n\frac{g_i^2}{a_i^2}\Big)^{1/2}\Big]
\ls \alpha_n\omega_n^{1/n}\cdot (n-1)\omega_{n-1}\frac{1}{a_1}\cdot \frac{1}{d_{n-1}}{\mathbb E}\Big[\Big (\sum_{i=2}^n\frac{g_i^2}{a_i^2}\Big)^{1/2}\Big].\end{equation*}
Since $x\mapsto \left (\sum_{i=1}^n\frac{x_i^2}{a_i^2}\right)^{1/2}$ is a seminorm, using H\"{o}lder and Khintchine's inequality 
for this seminorm in Gauss space we get
\begin{equation*}\frac{{\mathbb E}\left[\left (\sum_{i=1}^n\frac{g_i^2}{a_i^2}\right)^{1/2}\right]}{{\mathbb E}\left[\left  (\sum_{i=2}^n\frac{g_i^2}{a_i^2}\right)^{1/2}\right]}
\gr c\left (\frac{{\mathbb E}\Big(\sum_{i=1}^n\frac{g_i^2}{a_i^2}\Big)}{{\mathbb E}\Big (\sum_{i=2}^n\frac{g_i^2}{a_i^2}\Big)}\right )^{1/2}
= c\left (\frac{\sum_{i=1}^n\frac{1}{a_i^2}}{\sum_{i=2}^n\frac{1}{a_i^2}}\right )^{1/2},
\end{equation*}
and hence 
\begin{equation*}\alpha_n \gr c\frac{n\omega_n^{\frac{n-1}{n}}}{(n-1)\omega_{n-1}}\,\frac{d_{n-1}}{d_n}a_1\left (\frac{\sum_{i=1}^n\frac{1}{a_i^2}}{\sum_{i=2}^n\frac{1}{a_i^2}}\right )^{1/2}= c\frac{n\omega_n^{\frac{n-1}{n}}}{(n-1)\omega_{n-1}}\,\frac{d_{n-1}}{d_n}
\left (\frac{1+\sum_{i=2}^n\frac{a_1^2}{a_i^2}}{\sum_{i=2}^n\frac{1}{a_i^2}}\right )^{1/2}.
\end{equation*}
Now choose $a_2=\cdots =a_n=r$ and $a_1=r^{-(n-1)}$. Then,
$$\left (\frac{1+\sum_{i=2}^n\frac{a_1^2}{a_i^2}}{\sum_{i=2}^n\frac{1}{a_i^2}}\right )^{1/2}=
\left (\frac{1+\frac{n-1}{r^{2n}}}{\frac{n-1}{r^2}}\right )^{1/2}=\left (\frac{1}{r^{2n-2}}+\frac{r^2}{n-1}\right )^{1/2}\to\infty $$
as $r\to\infty $. So, we arrive at a contradiction, i.e. there can be no upper bound for $\alpha_n$. \end{proof}

\begin{remark}\rm Let us note here that a reverse inequality
can be obtained at least when $K$ is in some of the classical positions. It is proved in \cite{Giannopoulos-Hartzoulaki-Paouris-2002}
that for any convex body $K$ in ${\mathbb R}^n$ and any $\xi\in S^{n-1}$ we have
$$\frac{S(P_{\xi^{\perp }}(K))}{|P_{\xi^{\perp }}(K)|}\ls \frac{2(n-1)}{n}\frac{S(K)}{|K|},$$
therefore
$$|K|\,\max_{\xi\in S^{n-1}}S(P_{\xi^{\perp }}(K))\ls \frac{2(n-1)}{n}S(K)\max_{\xi\in S^{n-1}}|P_{\xi^{\perp }}(K)|.$$
Since we trivially have
$$|P_{\xi^{\perp }}(K)|=\frac{1}{2}\int_{S^{n-1}}|\langle \xi,\theta\rangle |\,d\sigma_K(\theta )\ls \frac{1}{2}S(K),$$
we see that
$$|K|\,\max_{\xi\in S^{n-1}}S(P_{\xi^{\perp }}(K))\ls \frac{n-1}{n}S(K)^2.$$
On the other hand, if $K$ is in some classical position (e.g. isotropic or John's position or minimal surface area or minimal mean width
position; see \cite[Chapter~2]{AGA-book}) then we know that a reverse isoperimetric inequality of the form $S(K)\ls cn|K|^{\frac{n-1}{n}}$
holds true (with an extra $\log n$-term in the minimal mean width position). Combining the above we see that, in this case,
$$|K|^{\frac{1}{n}}\,\max_{\xi\in S^{n-1}}S(P_{\xi^{\perp }}(K))\ls cn\,S(K)$$
for some absolute constant $c>0$.
\end{remark}

For the more general question, where surface area is replaced by any quermassintegral, we may exploit a formula from \cite{MTY-2020} for
the $j$-quermassintegrals of ellipsoids of revolution, i.e. ellipsoids of the form
$${\cal E}_{r,s}=\left\{x\in {\mathbb R}^m:\sum_{i=1}^{m-1}\frac{x_i^2}{r^2}+\frac{x_m^2}{s^2}\ls 1\right\}.$$
For every $j=0,1,\ldots ,m$ one has
\begin{equation}\label{eq:tatarko}W_j({\cal E}_{r,s})=\omega_mr^{m-j}\int_{S^{m-1}}
\Big(\frac{s^2}{r^2}\sum_{i=1}^{m-j}\theta_i^2+\sum_{i=m-j+1}^m\theta_i^2\Big)^{1/2}d\sigma (\theta ).\end{equation}

\begin{theorem}\label{th:negative-general}Let $n\gr 2$, $1\ls k\ls n$ and $0\ls j\ls n-k-1$. For every $\alpha >0$ there exists a convex body $K$ in ${\mathbb R}^n$ such that
\begin{equation*}W_j(K)>\alpha \,|K|^{\frac{k}{n}}\,\max_{F\in G_{n,n-k}}W_j(P_F(K))\gr \alpha \,|K|^{\frac{k}{n}}\,\max_{F\in G_{n,n-k}}W_j(K\cap F).\end{equation*}
\end{theorem}

\begin{proof}Assume that for some $n\gr 2$, $1\ls k\ls n$ and $0\ls j\ls n-k$ there exists a constant $C(n,k,j)>0$ such that
\begin{equation}\label{eq:general-false}W_j(K)\ls C(n,k,j)|K|^{\frac{k}{n}}\,\max_{F\in G_{n,n-k}}W_j(K\cap F).\end{equation}
Then, for any $r>1>s$ with $r^{n-1}s=1$ consider the ellipsoid ${\cal E}_{r,s}$. Recall that
$$\max_{F\in G_{n,n-k}}W_j({\cal E}_{r,s}\cap F)= W_j({\cal E}_{r,s}\cap F_{n-k}),$$
where $F_{n-k}=\mathrm{span}\{e_1,\ldots, e_{n-k}\}$. Note that $|{\cal E}_{r,s}|=\omega_n$ and that ${\cal E}_{r,s}\cap F_{n-k}$ is a ball
of radius $r$. Using
\eqref{eq:tatarko} and assuming that \eqref{eq:general-false} holds true, we see that
\begin{equation*}\omega_nr^{n-j}\int_{S^{n-1}}
\Big(\frac{s^2}{r^2}\sum_{i=1}^{n-j}\theta_i^2+\sum_{i=m-j+1}^m\theta_i^2\Big)^{1/2}d\sigma (\theta )
\ls C(n,k,j)\omega_n^{\frac{k}{n}}\omega_{n-k}r^{n-k-j}.\end{equation*}
Since
\begin{align*}\int_{S^{n-1}}\Big(\frac{s^2}{r^2}\sum_{i=1}^{n-j}\theta_i^2+\sum_{i=m-j+1}^m\theta_i^2\Big)^{1/2}d\sigma (\theta )
&\approx \Big(\int_{S^{n-1}}
\Big(\frac{s^2}{r^2}\sum_{i=1}^{n-j}\theta_i^2+\sum_{i=m-j+1}^m\theta_i^2\Big)d\sigma (\theta )\Big)^{1/2}\\
&=\Big(\frac{n-j}{n}\frac{s^2}{r^2}+\frac{j}{n}\Big)^{1/2}=\Big(\frac{n-j}{n}\frac{1}{r^{2n}}+\frac{j}{n}\Big)^{1/2},\end{align*}
we must have
$$r^k\Big(\frac{n-j}{n}\frac{1}{r^{2n}}+\frac{j}{n}\Big)^{1/2}\ls c_1C(n,k,j)\frac{\omega_{n-k}}{\omega_n^{\frac{n-k}{n}}}$$
for every $r>1$, which leads to a contradiction if we let $r\to\infty $. \end{proof}

\section{Bounds in terms of the parameter $t(K)$}

Let $K$ be a convex body in ${\mathbb R}^n$ with barycenter at the origin. Recall that $r(K)$ denotes the inradius of $K$; this
is the largest $r>0$ such that $x_0+rB_2^n\subseteq K$ for some $x_0\in K$. We also define the parameter
$$t(K):=\left (\frac{|K|}{|r(K)B_2^n|}\right )^{1/n}.$$
In this section we provide some positive results on the slicing problem for quermassintegrals, which however depend on $t(K)$.

\begin{theorem}\label{th:5.2.1}
Let $K$ be a convex body with barycenter at the origin in $\mathbb{R}^n$. Then, for every $1\ls j\ls n-k-1 \ls n-1$
we have that
$$W_j(K)\ls \alpha_{n,k,j}L_K^\frac{k(n-k-j)}{n-k}t(K)^j|K|^{\frac{k}{n}}\max_{H\in G_{n,n-k}}W_j(K\cap H),$$
where $\alpha_{n,k,j}=(\omega_n^\frac{j}{n}/\omega_{n-k}^\frac{j}{n-k})c^{\frac{n-k-j}{n-k}}$ and $c>0$ is an absolute constant.
\end{theorem}

\begin{proof}Using the monotonicity of mixed volumes we may write
$$W_j(K)=V((K,n-j),(B_2^n,j))\ls V\left((K,n-j),\left(\frac{K}{r(K)},j\right)\right)=\frac{1}{r(K)^j}V(K,\ldots ,K)=\frac{|K|}{r(K)^j}.$$
We rewrite this inequality in the form
\begin{equation}\label{eq:5.1}W_j(K)\ls\omega_n^\frac{j}{n}t(K)^j|K|^\frac{n-j}{n}=\omega_n^\frac{j}{n}t(K)^j|K|^\frac{k}{n}|K|^\frac{n-k-j}{n}.\end{equation}
Now, we use the estimate
$$\frac{c_0}{L_K}\le \widetilde{\Phi}_{[k]}(K):=\frac{1}{|K|^\frac{n-k}{nk}}\left(\int_{G_{n,n-k}}|K\cap H|^nd\nu_{n,n-k}\right)^\frac{1}{nk}$$
from \cite{Dafnis-Paouris-2012}. This gives
$$|K|^\frac{n-k}{nk}\ls \frac{L_K}{c_0}\left(\int_{G_{n,n-k}}|K\cap H|^nd\nu_{n,n-k}\right)^\frac{1}{nk}\ls
c_1L_K\max_{H\in G_{n,n-k}}|K\cap H|^\frac{1}{k},$$
where $c_1=1/c_0$, and hence,
$$|K|^\frac{n-k-j}{n}\ls (c_1L_K)^{\frac{k(n-k-j)}{n-k}}\max_{H\in G_{n,n-k}}|K\cap H|^\frac{n-k-j}{n-k}.$$
On the other hand, applying Aleksandrov's inequalities for $K\cap H$ we get
$$|K\cap H|^{\frac{n-k-j}{n-k}}\ls \omega_{n-k}^{-\frac{j}{n-k}}W_j(K\cap H)$$
for every $H\in G_{n,n-k}$. Combining the above we see that
$$|K|^\frac{n-k-j}{n}\ls \frac{1}{\omega_{n-k}^\frac{j}{n-k}}(c_1L_K)^{\frac{k(n-k-j)}{n-k}}\max_{H\in G_{n,n-k}}W_j(K\cap H),$$
and then \eqref{eq:5.1} takes the form
$$W_j(K)\ls (\omega_n^\frac{j}{n}/\omega_{n-k}^\frac{j}{n-k})(c_1L_K)^{\frac{k(n-k-j)}{n-k}}t(K)^j|K|^{\frac{k}{n}}\max_{H\in G_{n,n-k}}W_j(K\cap H).$$ Setting $\alpha_{n,k,j}=(\omega_n^\frac{j}{n}/\omega_{n-k}^\frac{j}{n-k})c_1^{\frac{k(n-k-j)}{n-k}}$ we conclude the proof.
\end{proof}

\begin{remark}\rm Let $d_s=s\omega_s^{1/s}$. In the particular case of surface area, we have the bounds
\begin{equation*}S(K)\ls \alpha_n\,L_K^{\frac{n-2}{n-1}}t(K)\,|K|^{\frac{1}{n}}\max_{\xi\in S^{n-1} }\,S(K\cap\xi^{\perp })\end{equation*}
for every $\xi\in S^{n-1}$, where $\alpha_n:=\frac{d_n}{d_{n-1}}(2\sqrt{3}e)^{\frac{n-2}{n-1}}$,
and more generally,
$$S(K)\ls \alpha_{n,k}L_K^\frac{k(n-k-1)}{n-k}t(K)|K|^{\frac{k}{n}}\max_{H\in G_{n,n-k}}S(K\cap H)$$
for every $1\ls k\ls n-1$, where $\alpha_{n,k}=\frac{d_n}{d_{n-k}}c^{\frac{k(n-k-1)}{n-k}}$.
\end{remark}

A variant of Theorem~\ref{th:5.2.1} is the following result.

\begin{theorem}\label{th:5.2.2}
Let $K$ be a convex body in $\mathbb{R}^n$ with $0\in {\rm int}(K)$. Then, for all $1\ls j\ls n-k\le n-1$ we have that
$$\frac{W_j(K)}{|K|}\ls t(K)^j\max_{H\in G_{n,n-k}}\frac{W_j(K\cap H)}{|K\cap H|}.$$
\end{theorem}

\begin{proof}Using the estimate $W_j(K)\ls |K|/r(K)^j$ we may write $W_j(K)\ls \omega_n^\frac{j}{n}t(K)^j|K|^{\frac{n-j}{n}}$,
therefore
\begin{equation}\label{eq:5.2.1}\frac{W_j(K)}{|K|}\ls \omega_n^\frac{j}{n}t(K)^j\frac{1}{|K|^\frac{j}{n}}.\end{equation}
Next, we use Grinberg's inequality
$$\min_{H\in G_{n,n-k}}|K\cap H|^n\ls\int_{G_{n,n-k}}|K\cap H|^nd\nu_{n,n-k}(F)\ls\frac{\omega_{n-k}^n}{\omega_n^{n-k}}|K|^{n-k}$$
to write
\begin{equation}\label{eq:5.2.2}\min_{H\in G_{n,n-k}}|K\cap H|^\frac{j}{n-k}\ls\frac{\omega_{n-k}^\frac{j}{n-k}}{\omega_n^\frac{j}{n}}|K|^\frac{j}{n}.\end{equation}
From Aleksandrov's inequality
$$|K\cap H|^{\frac{n-k-j}{n-k}}\ls \frac{W_j(K\cap H)}{\omega_{n-k}^\frac{j}{n-k}}$$
we see that
$$|K\cap H|^\frac{j}{n-k}\gr\omega_{n-k}^\frac{j}{n-k}\frac{|K\cap H|}{W_j(K\cap H)},$$ and hence
\begin{equation}\label{eq:5.2.3}\min_{{H\in G_{n,n-k}}}|K\cap H|^\frac{j}{n-k}\gr\omega_{n-k}^\frac{j}{n-k}\min_{H\in G_{n,n-k}}\frac{|K\cap H|}{W_j(K\cap H)}.\end{equation}
From \eqref{eq:5.2.2} and \eqref{eq:5.2.3} we get
$$\frac{\omega_n^\frac{j}{n}}{|K|^\frac{j}{n}}\ls \max_{H\in G_{n,n-k}}\frac{W_j(K\cap H)}{|K\cap H|},$$
and the theorem follows from \eqref{eq:5.2.1}.
\end{proof}

\begin{remark}\rm In the particular case of surface area, we have the bounds
\begin{equation*}\frac{S(K)}{|K|}\ls \frac{n}{n-1}t(K)\,\max_{\xi\in S^{n-1} }\,\frac{S(K\cap\xi^{\perp })}{|K\cap\xi^{\perp}|},\end{equation*}
and more generally,
$$\frac{S(K)}{|K|}\ls\frac{n}{n-k}t(K)\max_{H\in G_{n,n-k}}\frac{S(K\cap H)}{|K\cap H|}$$
for every $1\ls k\ls n-1$.
\end{remark}

\section{Bounds for the parameters $p(K)$ and $q(K)$}

In this section we discuss two parameters relating volume and surface area of a convex body.
Recall that if $r(K)$ is the radius of the largest Euclidean ball inscribed in $K$ then $x_0+r(K)B_2^n\subseteq K$ for some $x_0\in K$,
and hence we get
\begin{equation*}r(K)S(K) =nV(K,\ldots ,K, x_0+r(K)B_2^n)\ls nV(K,\ldots ,K,K)=n|K|.\end{equation*}
by the monotonicity and translation invariance of mixed volumes. On the other hand, if $R(K)=\max\{h_K(\xi ):\xi\in S^{n-1}\}$ is the radius of $K$ then the formula
$$n|K|=\int_{S^{n-1}}h_K(\xi )d\sigma_K(\xi )$$
where $\sigma_K$ is the surface area measure of $K$, implies that
\begin{equation*}n|K|\ls R(K)\sigma_K(S^{n-1})=R(K)S(K).\end{equation*}
Starting from the observation we did in the introduction that for an ellipsoid ${\cal E}$ in ${\mathbb R}^n$ we have the precise formula
$$S({\cal E})\approx n\,|{\cal E}|M({\cal E})$$
where  $M({\cal E})=\int_{S^{n-1}}\|\xi\|_{{\cal E}}d\sigma (\xi )$,
and the fact that $r(K)\ls\frac{1}{M(K)}\ls w(K)\ls R(K)$, it is natural to introduce the parameters
$$p(K)=\frac{S(K)}{|K|M(K)}\quad\hbox{and}\quad q(K)=\frac{w(K)S(K)}{|K|}.$$
Our aim is to provide optimal upper and lower bounds for $p(K)$ and $q(K)$, both in general and in the
case where $K$ is in some of the classical positions.

Starting with $p(K)$, we show in Theorem~\ref{th:2.1} and Theorem~\ref{parameter_lower} below that there exist absolute constants $c_1, c_2>0$ such that for every convex body $K\in\mathbb{R}^n$ we have $c_1\sqrt{n}\ls p(K)\ls c_2n^{3/2}$. Moreover, both estimates give the optimal dependence on the dimension.

\begin{theorem}\label{th:2.1}
Let $K$ be a centered convex body in $\mathbb{R}^n$. Then $S(K)\ls cn^{3/2}|K|M(K)$, where $c$ is an absolute constant. In fact,
$\sup\{p (K): K\;\hbox{a centered convex body in}\;{\mathbb R}^n\}\approx n^{3/2}$.
\end{theorem}

\begin{proof}First we consider the centrally symmetric case. By the simple case $k=1$ of the Rogers-Shephard inequality we have that
$n|K|\gr |K\cap \langle u\rangle|\,|P_{u^{\perp}}K|=2\rho_K(u)\,|P_{u^{\perp}}K|$, which may be written as
$$\frac{n}{2}\|u\|_K|K|\gr |P_{u^{\perp}}K|.$$
Now we integrate over the sphere and using Cauchy's surface area formula
$$\int_{S^{n-1}}|P_{u^{\perp }}(K)|\,d(\sigma\theta )=\frac{\omega_{n-1}}{n\omega_n}S(K)$$
we get
$$\frac{n}{2}M(K)|K|\gr\frac{\omega_{n-1}}{n\omega_n}S(K)\approx\frac{1}{\sqrt{n}}S(K).$$
For the general case we may use a different argument. Since $r(K)S(K)\ls n|K|$,
it is enough to check that $r(K)M(K)\gr c/\sqrt{n}$ for some absolute constant $c>0$.
Indeed, passing to the polars, it suffices to prove that
$${\rm diam}(K)\ls c\sqrt{n}w(K).$$
To prove this, observe that $K$ contains a segment $I$ with length equal to $\rm{diam}(K)$. Since the value of the mean width depends on whether $K$ lives in a subspace of $\mathbb{R}^n$ or not, we compute the Gaussian mean-width $w_G(K):=\int_{{\mathbb R}^n}h_K(x)d\gamma_n(x)$,
where $\gamma_n$ is the standard Gaussian measure on ${\mathbb R}^n$, which does not depend on it. Integration in polar coordinates shows that
$$\sqrt{n}w(K)\approx w_G(K)\gr w_G(I)\approx {\rm diam}(K).$$
To see why this upper bound is sharp, consider the family of polyhedra
$$P_{s}=\left\{x\in\mathbb{R}^n\colon\, |x_1|+\frac{1}{s}\sum_{2}^{n}|x_i|\ls 1\right\},$$
where $s>0$. The distance from the origin to each facet of $P_{s }$ is equal to
$$r(P_s)=\frac{1}{\sqrt{1+\frac{n-1}{s^2}}}.$$
Note also that $|P_s|=\frac{1}{n}S(P_s)r(P_s)$. On the other hand,
$$M(P_s)=\int_{S^{n-1}}\Big(|\theta_1|+\frac{1}{s}\sum_{2}^{n}|\theta_i|\Big)\,d\sigma (\theta )\approx \frac{1}{\sqrt{n}}\Big(1+\frac{n-1}{s}\Big).$$
Therefore,
$$p (P_s)=\frac{S(P_{s})}{|P_{s}|M(P_s)}=\frac{n}{r(P_{s})M(P_s)}
\approx\frac{n^{3/2}\sqrt{1+\frac{n-1}{s^2}}}{1+\frac{n-1}{s}}.$$
Since $$\lim_{s\to\infty }\frac{\sqrt{1+\frac{n-1}{s^2}}}{1+\frac{n-1}{s}}=1,$$
the result follows.
\end{proof}

\begin{theorem}\label{thm:upper_1}
Let $K$ be a convex body in $\mathbb{R}^n$ and let $r(K)$ denote the radius of the largest Euclidean ball inscribed in $K$. Then,
$S(K)\gr\frac{|K|}{r(K)}$. In fact, $\inf\left\{\frac{S(K)r(K)}{|K|}: K\;\hbox{a convex body in}\;{\mathbb R}^n\right\}= 1$.
\end{theorem}

\begin{proof}
The first inequality is a consequence of the following Bonnesen-type inequality that can be found in \cite{Osserman}:
$$S(K)\gr \frac{|K|}{r(K)}+(n-1)\omega_nr(K)^{n-1}.$$
To check the second assertion of the theorem, consider the family of parallelepipeds
$$P_{a,s}=\{x:\, |x_1|\ls s, |x_i|\ls a\ \text{for}\ i\gr 2 \}$$ where $0<s<a$. Then, $r(P_{a,s})=s$, $|P_{a,s}|=2^nsa^{n-1}$ and
$S(P_{a,s})=2^na^{n-1}+2^n(n-1)sa^{n-2}$. Letting $a\rightarrow\infty$ gives
$$\lim_{a\to\infty }\frac{S(P_{a,s})r(P_{a,s})}{|P_{a,s}|}=\lim_{a\to\infty }\frac{2^ns(a^{n-1}+(n-1)a^{n-2}s)}{2^nsa^{n-1}}
=1,$$
and the result follows.
\end{proof}

\begin{theorem}\label{parameter_lower}
Let $K$ be a centrally symmetric convex body in $\mathbb{R}^n$. Then
$$S(K)\gr\sqrt{n}|K|M(K).$$
In fact, $\inf\{p (K): K\;\hbox{a centered convex body in}\;{\mathbb R}^n\}\approx \sqrt{n}$.
\end{theorem}

\begin{proof}From the Rogers-Shephard inequality we have that
$$|K|\ls |K\cap \langle u\rangle||P_{u^{\perp}}K|=\frac{2}{\|u\|_K}|P_{u^{\perp}}K|,$$
which gives that
$$\frac{\|u\|_K}{2}|K|\ls |P_{u^{\perp}}K|.$$
Now we integrate over the sphere to get
$$\frac{1}{2}M(K)|K|\ls\frac{\omega_{n-1}}{n\omega_n}S(K)\approx\frac{1}{\sqrt{n}}S(K).$$
This bound cannot be improved in general. To see this, first recall that for every symmetric convex body
$D$ in ${\mathbb R}^n$ we have $\frac{1}{r(D)}\ls c\sqrt{n}M(D)$. Then, observe that for the parallelepipeds $P_{a,s}$ in
the proof of Theorem~\ref{thm:upper_1} we have that
$$p (P_{a,s})=\frac{S(P_{a,s})}{|P_{a,s}|M(P_{a,s})}\ls c\sqrt{n}\frac{S(P_{a,s})r(P_s)}{|P_{a,s}|}$$
and hence
$$\inf\{p (K): K\;\hbox{a centered convex body in}\;{\mathbb R}^n\}\ls\lim_{a\to\infty }p (P_{a,s})
\ls c\sqrt{n}\lim_{a\to\infty }\frac{S(P_{a,s})r(P_s)}{|P_{a,s}|}=c\sqrt{n},$$
which completes the proof.
\end{proof}

Assume that $K$ is in John's position; this means that the Euclidean unit ball $B_2^n$ is the ellipsoid of maximal volume
which is inscribed in $K$. In this case, one can get a better estimate for $p (K)$, which is actually sharp as one can check from the
example of the cube $Q_n=[-1,1]^n$; note that $S(Q_n)=2n\cdot 2^{n-1}$ and $M(Q_n)\approx \sqrt{\log n}/\sqrt{n}$, therefore
$$p (Q_n)=\frac{n}{M(Q_n)}\approx \frac{n^{3/2}}{\sqrt{\log n}}.$$

\begin{theorem}
Let $K$ be a convex body in $\mathbb{R}^n$ which is in John's position. Then,
$$S(K)\ls c\frac{n^{3/2}}{\sqrt{\log n}}|K|M(K),$$
where $c$ is an absolute constant.
\end{theorem}

\begin{proof}
Since $B_2^n\subseteq K$ we have $S(K)=nV(K,\ldots ,K,B_2^n)\ls n|K|$. Schmuckenschl\"{a}ger has proved in \cite{Schmuckenschlager-1999} (see also \cite{Barthe-1998} for the dual result) that
$M(K)\gr M(\Delta_n)$, where $\Delta_n$ is a regular simplex in John's position. Moreover,
$$M(\Delta_n)\gr c\frac{\sqrt{\log n}}{\sqrt{n}},$$ and the result follows.
\end{proof}

The next result provides some bounds for $p (K)$ when $K$ is in the minimal surface position.

\begin{theorem}
Let $K\subseteq\mathbb{R}^n$ be a centrally symmetric convex body in minimal surface area position. Then,
$$S(K)\leqslant\frac{n}{M(B_\infty^n)}|K|M(K).$$
\end{theorem}

\begin{proof}
Let us assume that $K$  is a centrally symmetric polytope in minimal surface area position, with facets $\{F_j\}_{j=1}^m$ and outer normal vectors $\{u_j\}_{j=1}^m$ Then,
$$K=\{x\in\mathbb{R}^n\,:\,|\langle x, u_j\rangle|\leqslant h_K(u_j),\,1\leqslant j\leqslant m\}$$
and
$$I_n=\sum_{j=1}^m\frac{n|F_j|}{S(K)}u_j\otimes u_j=\sum_{j=1}^mc_ju_j\otimes u_j,$$
where $c_j=\frac{n|F_j|}{S(K)}$ for every $1\leqslant j\leqslant m$ (see \cite[Chapter~2]{AGA-book}).

Let $t\gr 0$. Using the Brascamp-Lieb inequality (see \cite[Chapter~2]{AGA-book}) we get
\begin{align*}
\gamma_n(tK)&=\int_{\mathbb{R}^n}\prod_{j=1}^m \chi_{[-th_K(u_j),th_K(u_j)]}(\langle x, u_j\rangle)\frac{e^{-\sum_{j=1}^m\frac{c_j\langle x,u_j\rangle^2}{2}}}{(2\pi )^{n/2}}dx\\
&\leqslant\prod_{j=1}^m \left(\int_{-th_K(u_j)}^{th_K(u_j)}\frac{e^{-\frac{s^2}{2}}}{\sqrt{2\pi }}ds\right)^{c_j}=\prod_{j=1}^m \gamma_1\left(\left[-th_K(u_j),th_K(u_j)\right]\right)^{c_j}.
\end{align*}
Since the 1-dimensional Gaussian measure is log-concave, we have that
\begin{equation*}
\prod_{j=1}^m \gamma_1\left(\left[-th_K(u_j),th_K(u_j)\right]\right)^{c_j}
\leqslant\gamma_1\Big(\Big(\sum_{j=1}^m \frac{tc_jh_K(u_j)}{n}\Big)[-e_1,e_1]\Big)^n=\gamma_n\left(tn\frac{|K|}{S(K)}B_\infty^n\right).
\end{equation*}
Therefore, for any $t\gr 0$ we have that
$$\gamma_n(tK)\leqslant \gamma_n\left(tn\frac{|K|}{S(K)}B_\infty^n\right),$$
and since $M(K)$ is a multiple of the integral $\int_0^{\infty }(1-\gamma_n(tK))dt$ this implies that
$$M(K)\geqslant\frac{1}{n}\frac{S(K)}{|K|}M(B_\infty^n),$$
which is the assertion of the theorem.
\end{proof}

Our last result relates $p (K)$ with the average section parameter $\mathrm{as}(K)$, defined by
$$\mathrm{as}(K)=\int_{S^{n-1}}|K\cap \xi^{\perp}|\, d\sigma(\xi ).$$

\begin{theorem}\label{thm:upper_2}
Let $K$ be a convex body in $\mathbb{R}^n$. Then, $S(K)\gr\frac{n\omega_n}{\omega_{n-1}}\mathrm{as}(K)$.
The estimate is sharp when $K=B_2^n$.
\end{theorem}

\begin{proof}
We write $S(K)=nW_1(K)=nV(K,\ldots, K, B_2^n)$ and $$\mathrm{as}(K)=\omega_{n-1}\int_{S^{n-1}}\rho_K^{n-1}(\xi )\, d\sigma(\xi )=\frac{\omega_{n-1}}{\omega_n}\tilde{V}(K,\ldots, K, B_2^n).$$
Therefore, the inequality $S(K)\gr\frac{n\omega_n}{\omega_{n-1}}\mathrm{as}(K)$ is equivalent to
$$V(K,\ldots, K, B_2^n)\gr \tilde{V}(K,\ldots, K, B_2^n),$$
which is true by Corollary~1.3 in \cite{Lutwak-1975}.
\end{proof}

Since in the isotropic position we have that $|K\cap u^{\perp}|\approx\frac{1}{L_K}$, and we also know that
$M(K)\ls\frac{c}{n^{\epsilon }L_K}$ for some $\epsilon >0$ (the currently best known estimate is with $\epsilon =1/10$, due to
Giannopoulos and E.~Milman, see \cite{Giannopoulos-EMilman-2014}) we get

\begin{corollary}
Let $K$ be an isotropic convex body in $\mathbb{R}^n$. Then,
$$S(K)\gr\frac{c\sqrt{n}}{L_K}\gr n^{1/2+\epsilon}|K|M(K),$$
where $\epsilon \gr\frac{1}{10}$. Therefore, $p (K)\gr cn^{1/2+\epsilon }$.
\end{corollary}

\begin{note*}The only general lower bound that we know for $\mathrm{as}(K)$ is
$\mathrm{as}(K)\gr c\sqrt{n}\frac{|K|}{R(K)}$, from \cite{BGKD}.\end{note*}

\smallskip

The situation is different with the parameter $q(K)$. Regarding the lower bound, if we combine the isoperimetric inequality
$S(K)\gr n\omega_n^{1/n}|K|^{\frac{n-1}{n}}$ with Urysohn's inequality $w(K)\gr {\rm vrad}(K)$, we readily see that
$$w(K)S(K)\gr\frac{|K|^{\frac{1}{n}}}{\omega_n^{1/n}}\,n\omega_n^{1/n}|K|^{\frac{n-1}{n}}=n|K|.$$
Therefore, $q(K)\gr n$ for every convex body $K$ in ${\mathbb R}^n$ with equality if $K=B_2^n$ is the Euclidean unit ball.
However, we observe that for any fixed dimension $n\gr 2$ there is no upper bound for $q(K)$.

\begin{theorem}\label{th:negative-q}For any $n\gr 2$ one has that $q(K)\gr n$ and
$$\sup\{q(K):K\;\mathrm{is \;a \;centrally\;symmetric\;convex\; body\; in}\;{\mathbb R}^n\}=+\infty .$$
\end{theorem}

\begin{proof}We have explained the first assertion and for the second one we consider, again, the class of ellipsoids.
Let ${\cal E}$ be an ellipsoid in ${\mathbb R}^n$ with semi-axes $a_1\ls a_2\ls\cdots \ls a_n$ in the directions of $e_1,\ldots ,e_n$.
We may assume that $\prod_{i=1}^na_i=1$. Recall that $S({\cal E})\approx n|{\cal E}|M({\cal E})$, and hence
$$q({\cal E})=\frac{w({\cal E})S({\cal E})}{|{\cal E}|}\approx nw({\cal E})M({\cal E}).$$
Now,
$$M({\cal E})=\int_{S^{n-1}}\Big (\sum_{i=1}^n\frac{\xi_i^2}{a_i^2}\Big)^{1/2}d\sigma (\xi )$$
and $$w({\cal E})=M({\cal E}^{\circ })=\int_{S^{n-1}}\Big (\sum_{i=1}^n\xi_i^2a_i^2\Big)^{1/2}d\sigma (\xi ).$$
It follows that
$$q({\cal E})\gr cn\int_{S^{n-1}}\frac{|\xi_1|}{a_1}d\sigma (\xi )\cdot\int_{S^{n-1}}|\xi_n|a_nd\sigma (\xi )\approx \frac{a_n}{a_1}.$$
It is clear that if we choose $a_1=\cdots =a_{n-1}=a<1$ and $a_n=1/a^{n-1}$ then $\prod_{i=1}^na_i=1$ and $q({\cal E})\gr c/a^n\to\infty $
as $a\to 0^+$, which proves our claim. \end{proof}

\section{Isomorphic Busemann-Petty problem for the surface area}

Busemann-Petty type problems for surface area are closely related to the questions we address in this article. 
A first question that may be asked is if an origin-symmetric convex body is uniquely determined by the surface area of its hyperplane
central sections; it is well-known (see \cite{RYZ-survey}) that origin-symmetric star bodies are uniquely determined by the volume of their central sections. In its simplest form, when $n = 3$, this question is asked by Gardner in his book \cite{Gardner-book}: If $K$ and $L$
are two origin-symmetric convex bodies in $\mathbb{R}^3$ such that the sections $K\cap\xi^{\perp}$ and $L\cap\xi^{\perp}$ have equal perimeters for all $\xi\in S^{2}$ is it then true that $K=L$? To the best of our knowledge, the problem is open in full generality.
An affirmative answer is given in \cite{Howard-Naza-Rya-Zva} for the class of $C^1$ star bodies of revolution, and an infinitesimal 
version of the problem is settled in \cite{rusu-thesis} where it is shown that the answer is affirmative if one of the bodies is the Euclidean ball 
and the other is its one parameter analytic deformation. Yaskin has proved in \cite{Yaskin} that the answer is affirmative for 
the class of origin-symmetric convex polytopes in $\mathbb{R}^n$, where in dimensions $n\gr 4$ the perimeter is replaced by the surface 
area of the sections.

The analogue of the Busemann-Petty problem for surface area was studied by Koldobsky and K\"{o}nig in \cite{KK-2019}: 
If $K$ and $D$ are two convex bodies in ${\mathbb R}^n$ such that
$S(K\cap\xi^{\perp })\ls S(D\cap\xi^{\perp })$ for all $\xi\in S^{n-1}$ does it then follow that $S(K)\ls S(D)$? Answering a
question of Pe\l czynski, they prove that the central $(n-1)$-dimensional section of the cube $B_{\infty }^n=[-1,1]^n$ that has
maximal surface area is the one that corresponds to the unit vector $\xi_0=\frac{1}{\sqrt{2}}(1,1,0,\ldots ,0)$ (exactly as in
the case of volume) i.e.
$$\max_{\xi\in S^{n-1}}S(B_{\infty }^n\cap\xi^{\perp })=S(B_{\infty }^n\cap\xi_0^{\perp })=2((n-2)\sqrt{2}+1).$$
Comparing with a ball of suitable radius one gets that the answer to the Busemann-Petty problem for surface area is negative
in dimensions $n\gr 14$.
 It is natural to ask whether an isomorphic version of the problem has an affirmative answer. This corresponds
to finding a constant $\beta_n$ (possibly independent from the dimension $n$) such that if $K$ and $D$ are two convex bodies
in ${\mathbb R}^n$ with $S(K\cap\xi^{\perp })\ls S(D\cap\xi^{\perp })$ for all $\xi\in S^{n-1}$ then $S(K)\ls \beta_nS(D)$.

Starting from the equivalence of the isomorphic Busemann--Petty problem with the slicing problem, one may think of the
corresponding connection if we consider surface area in place of volume. Suppose that there is a constant $\gamma_n$ such that
if $K$ and $D$ are centrally symmetric convex bodies in $\mathbb{R}^n$ that satisfy
$$S(K\cap\xi^{\perp})\ls S(D\cap\xi^{\perp}),$$ for all $\xi\in S^{n-1}$, then $S(K)\ls \gamma_n S(D)$.
Now, let $K$ be a convex body in $\mathbb{R}^n$ and choose $\xi_0\in S^{n-1}$ such that
$$S(K\cap\xi_0^{\perp})=\max_{\xi\in S^{n-1}} S(K\cap\xi^{\perp})$$
and $r>0$ such that $r^{n-2}S(B_2^{n-1})=S(rB_2^{n-1})=S(K\cap\xi_0^{\perp})$. Then,
$$S(K\cap\xi^{\perp})\ls S(rB_2^{n}\cap\xi^{\perp}),$$ for all $\xi\in S^{n-1}$. Therefore,
$$S(K)^{\frac{n-2}{n-1}}\ls \gamma_n^{\frac{n-2}{n-1}}S(rB_2^n)^{\frac{n-2}{n-1}}=\gamma_n^{\frac{n-2}{n-1}}\frac{S(B_2^n)^{\frac{n-2}{n-1}}}{S(B_2^{n-1})}\max_{\xi\in S^{n-1}} S(K\cap\xi^{\perp}).$$
This implies that there is some constant $c(n)$ such that
\begin{equation}\label{slicing-1}
S(K)\ls c(n)S(K)^{\frac{1}{n-1}}\max_{\xi\in S^{n-1}} S(K\cap\xi^{\perp}).
\end{equation}
The validity of \eqref{slicing-1} is a new question, which is of course related to the question that we discuss in this article.

We start with an estimate for ellipsoids.

\begin{proposition}\label{sec-5-1}Let ${\cal E}$ be an origin symmetric ellipsoid in ${\mathbb R}^n$. Then,
$$\frac{S(\mathcal{E})}{\max\limits_{\xi\in S^{n-1}}S(\mathcal{E}\cap\xi^{\perp })}\ls D_nr(\mathcal{E})^{-\frac{1}{n-1}}$$
where $D_n>0$ is bounded by an absolute constant.
\end{proposition}

\begin{proof}We may assume that $|{\cal E}|=1$. Let $a_1\ls\cdots\ls a_n$ be the lengths of its principal semi-axes
of ${\cal E}$ in the directions of $e_1,\ldots ,e_n$. We have seen that
$$\max_{\xi }S({\cal E}\cap\xi^{\perp })=S({\cal E}\cap e_1^{\perp })=(n-1)\,|{\cal E}\cap e_1^{\perp }|\,\int_{S^{n-2}}\Big (\sum_{i=2}^n\frac{\xi_i^2}{a_i^2}\Big)^{1/2}d\sigma (\xi ).$$
Then,
$$\frac{S(\mathcal{E})}{\max\limits_{\xi\in S^{n-1}}S(\mathcal{E}\cap\xi^{\perp })}=C_na_1\frac{\mathbb{E}\left[\left  (\sum_{i=1}^n\frac{g_i^2}{a_i^2}\right)^{1/2}\right]}{\mathbb{E}\left[\left (\sum_{i=2}^n\frac{g_i^2}{a_i^2}\right)^{1/2}\right]},$$
where $C_n$ is bounded by an absolute constant.

Since
\begin{equation*}\frac{\mathbb{E}\left[\left  (\sum_{i=1}^n\frac{g_i^2}{a_i^2}\right)^{1/2}\right]}{\mathbb{E}\left[\left (\sum_{i=2}^n\frac{g_i^2}{a_i^2}\right)^{1/2}\right]}
\ls c\left (\frac{{\mathbb E}\Big(\sum_{i=1}^n\frac{g_i^2}{a_i^2}\Big)}{{\mathbb E}\Big (\sum_{i=2}^n\frac{g_i^2}{a_i^2}\Big)}\right )^{1/2}
= c\left (\frac{\sum_{i=1}^n\frac{1}{a_i^2}}{\sum_{i=2}^n\frac{1}{a_i^2}}\right )^{1/2},
\end{equation*}
we have that $$\frac{S(\mathcal{E})}{\max\limits_{\xi\in S^{n-1}}S(\mathcal{E}\cap\xi^{\perp })}\ls C_na_1\left (\frac{\sum_{i=1}^n\frac{1}{a_i^2}}{\sum_{i=2}^n\frac{1}{a_i^2}}\right )^{1/2}=C_n\left(1+\frac{1}{\sum_{i=2}^n\frac{a_1^2}{a_i^2}}\right )^{1/2}.$$
Using the arithmetic-geometric mean inequality we get
\begin{align*}
\sum_{i=2}^n\frac{a_1^2}{a_i^2}\gr (n-1)a_1^2\left(\frac{1}{a_2^2\ldots a_n^2}\right)^{\frac{1}{n-1}}=(n-1)a_1^2a_1^{\frac{2}{n-1}}=(n-1)a_1^{\frac{2n}{n-1}}.
\end{align*}
Moreover, $1\ls\frac{1}{a_1^{\frac{2n}{n-1}}}$ and adding these two inequalities we get
$$\left(1+\frac{1}{\sum_{i=2}^n\frac{a_1^2}{a_i^2}}\right )^{1/2}\ls
\left(\frac{1}{a_1^{\frac{2n}{n-1}}}+\frac{1}{(n-1)a_1^{\frac{2n}{n-1}}}\right)^{\frac{1}{2}},$$ therefore
$$\frac{S(\mathcal{E})}{\max\limits_{\xi\in S^{n-1}}S(\mathcal{E}\cap\xi^{\perp })}\ls D_n\frac{1}{a_1^{\frac{1}{n-1}}}=D_n\frac{1}{r(\mathcal{E})^{\frac{1}{n-1}}},$$
where $D_n$ is bounded by an absolute constant.
\end{proof}

\begin{remark}\rm The example of an ellipsoid $\mathcal{F}$ with $a_2=\ldots=a_n=r$ and $a_1=\frac{1}{r^{n-1}}$ gives that
$$\frac{S(\mathcal{F})}{\max_{\xi\in S^{n-1}}S(\mathcal{F}\cap\xi^{\perp })}\gr E_n\frac{1}{r(\mathcal{F})^{\frac{1}{n-1}}},$$
therefore the inequality of Proposition~\ref{sec-5-1} is sharp.
\end{remark}

From Theorem~\ref{thm:upper_1} we know that $\frac{1}{r(K)}\ls S(K)$ for every convex body $K$ of volume $1$ in ${\mathbb R}^n$.
Combining this fact with Proposition~\ref{sec-5-1} we immediately get the next theorem which confirms \eqref{slicing-1} for the class of ellipsoids.

\begin{theorem}\label{sec-5-2}Let ${\cal E}$ be an origin symmetric ellipsoid in ${\mathbb R}^n$.
Then,
$$S(\mathcal{E})\ls A_nS(\mathcal{E})^{\frac{1}{n-1}}\max\limits_{\xi\in S^{n-1}}S(\mathcal{E}\cap\xi^{\perp })$$
where $A_n>0$ is bounded by an absolute constant.
\end{theorem}

Using John's theorem and the monotonicity of surface area one can easily deduce that a similar estimate holds
true in full generality: For any convex body $K$ in ${\mathbb R}^n$ one has
$$S(K)\ls A_n^{\prime }S(K)^{\frac{1}{n-1}}\max\limits_{\xi\in S^{n-1}}S(K\cap\xi^{\perp })$$
where $A_n^{\prime }>0$ is a constant depending only on $n$. It is an interesting question to determine the
best possible behavior of the constant $A_n^{\prime }$ with respect to the dimension $n$.

\bigskip

\noindent {\bf Acknowledgements:} We would like to thank A.~Giannopoulos and A.~Koldobsky for helpful discussions
We acknowledge support by the Hellenic Foundation for Research and Innovation (H.F.R.I.) under the ``First Call for H.F.R.I.
Research Projects to support Faculty members and Researchers and the procurement of high-cost research equipment grant" (Project
Number: 1849).

\bigskip

\bigskip

\footnotesize
\bibliographystyle{amsplain}

\begin{thebibliography}{100}
\footnotesize
\bibitem{AGA-book}
\textrm{S.\ Artstein-Avidan, A.\ Giannopoulos and V.\ D.\ Milman},
\textit{Asymptotic Geometric Analysis, Part I}, Mathematical Surveys and Monographs {\bf 202}, Amer. Math. Soc., Providence RI (2015).

\bibitem{Barthe-1998} {\rm F.\ Barthe}, \textit{An extremal property of
the mean width of the simplex}, Math. Ann. {\bf 310} (1998), 685--693.

\bibitem{Bourgain-1991} {\rm J.\ Bourgain}, \textit{On the distribution of polynomials on high dimensional convex sets}, in Geom. Aspects of Funct.
Analysis, Lecture Notes in Mathematics {\bf 1469}, Springer, Berlin (1991), 127--137.

\bibitem{BGVV-book} \textrm{S.\ Brazitikos, A.\ Giannopoulos, P.\ Valettas and B-H.\ Vritsiou},
\textit{Geometry of isotropic convex bodies}, Mathematical Surveys and Monographs {\bf 196}, Amer. Math. Soc., Providence RI (2014).

\bibitem{BGKD}
\textrm{S.\ Brazitikos, S.\ Dann, A.\ Giannopoulos and A.\ Koldobsky}, \textit{On the average volume of sections of convex bodies}, Isr. J. Math. {\bf 222}
(2017), 921--947.

\bibitem{Burago-Zalgaller-book} {\rm Y.\ D.\ Burago and V.\ A.\ Zalgaller}, \textit{Geometric Inequalities}, Springer Series in
Soviet Mathematics, Springer-Verlag, Berlin-New York (1988).

\bibitem{BP} {H.~Busemann and C.~M.~Petty,} {\it Problems on convex bodies}, Math. Scand. {4} (1956), 88--94.
\bibitem{Busemann-Straus-1960} {H.~Busemann and E.~G.~Straus}, \textit{Area and Normality},
Pacific J.\ Math. {\bf 10} (1960), 35--72.
\bibitem{CGL} {G.~Chasapis, A.~Giannopoulos and D.~Liakopoulos},
{\it Estimates for measures of lower dimensional sections of convex bodies}, Adv. Math. {\bf 306}
(2017), 880--904.

\bibitem{YChen} {\rm Y.~Chen}, \textit{An Almost Constant Lower Bound of the Isoperimetric Coefficient in the KLS Conjecture},
Geom. Funct. Anal. {\bf 31} (2021), no. 1, 34--61.
\bibitem{Dafnis-Paouris-2010} {\rm N.\ Dafnis and G.\ Paouris}, \textit{Small ball probability estimates,
$\psi_2$-behavior and the hyperplane conjecture}, J.\ Funct.\ Anal.\ {\bf 258} (2010), 1933--1964.
\bibitem{Dafnis-Paouris-2012} \textrm{N.\ Dafnis and G.\ Paouris}, \textit{Estimates for the affine and dual
affine quermassintegrals of convex bodies}, Illinois J.\ of Math.\ \textbf{56} (2012), 1005--1021.

\bibitem{Gardner-book} {R.~J.~Gardner}, \textit{Geometric tomography}, Second edition,
Cambridge University Press, Cambridge, (2006).
\bibitem{Giannopoulos-EMilman-2014} {\rm A. Giannopoulos and E. Milman}, {\sl $M$-estimates for isotropic convex bodies
and their $L_q$ centroid bodies}, in Geometric Aspects of Functional Analysis, Lecture Notes in Mathematics {\bf 2116} (2014), 159-182.
\bibitem{Giannopoulos-Hartzoulaki-Paouris-2002} {\rm A. Giannopoulos, M. Hartzoulaki and
G. Paouris}, {\sl On a local version of the Aleksandrov-Fenchel inequalities
for the quermassintegrals of a convex body}, Proc. Amer. Math. Soc. {\bf 130} (2002), 2403-2412.
\bibitem{GKV} {\rm A.\ Giannopoulos, A.\ Koldobsky and P.\ Valettas},
\textit{Inequalities for the surface area of projections of convex bodies},  Canad. J. Math. {\bf 70} (2018), no. 4, 804–-823.
\bibitem{Grinberg-1990} {\rm E.~L.~Grinberg}, \textit{Isoperimetric inequalities and identities for $k$-dimensional cross-sections of
convex bodies}, Math. Ann. {291} (1991), 75--86.

\bibitem{Howard-Naza-Rya-Zva}{\rm R.~Howard, F.~Nazarov, D.~Ryabogin and A.~Zvavitch}, \textit{Determining starlike bodies
	by the perimeters of their central sections}, preprint.

\bibitem{Kab-Zap-2016} {\rm Z.\ Kabluchko and D.\ Zaporozhets}, \textit{Intrinsic volumes of Sobolev balls with applications to Brownian
convex hulls},  Trans. Amer. Math. Soc. {\bf 368} (2016), no. 12, 8873--8899.
\bibitem{Klartag-2006} \textrm{B.\ Klartag}, \textit{On convex perturbations with a bounded isotropic constant},
Geom.\ Funct.\ Anal.\ {\bf 16} (2006), 1274--1290.
\bibitem{Klartag-VMilman-2005b} {\rm B.\ Klartag and V. D.\ Milman},
\textit{Rapid Steiner symmetrization of most of a convex body and the
slicing problem}, Combin. Probab. Comput. {\bf 14} (2005), 829--843.
\bibitem{Koldobsky-book} {A.~Koldobsky}, \textit{Fourier analysis in convex geometry},
Mathematical Surveys and Monographs {\bf 116}, Amer. Math. Soc., Providence RI (2005).

\bibitem{Koldobsky-2013} {\rm A.\ Koldobsky}, \textit{Stability and separation in volume comparison problems},
Math. Model. Nat. Phenom. {\bf 8} (2013), 156--169.

\bibitem{Koldobsky-2015} {\rm A.\ Koldobsky}, \textit{Stability inequalities for projections of convex bodies},
Discrete Comput. Geom. {\bf 57} (2017), 152--163.

\bibitem{Koldobsky-Advances-2014} {\rm A.\ Koldobsky}, \textit{A $\sqrt{n}$-estimate for measures of hyperplane
sections of convex bodies}, Adv.\ Math.\ {\bf 254} (2014), 33--40.

\bibitem{Koldobsky-GAFA-2014} {\rm A.\ Koldobsky}, \textit{Estimates for measures of sections of convex bodies},
in Geometric Aspects of Functional Analysis, Lecture Notes in Mathematics {\bf 2116} (2014), 261--271.

\bibitem{Koldobsky-Advances-2015} {\rm A.\ Koldobsky}, \textit{Slicing inequalities for measures of convex bodies},
Adv.\ Math.\ {\bf 283} (2015), 473--488.

\bibitem{Koldobsky-private} {\rm A.\ Koldobsky}, \textit{Private communication}.

\bibitem{KK-2019} {\rm A.\ Koldobsky and H.\ K\"{o}nig}, \textit{On the maximal perimeter of sections of the cube}, Adv. Math. {\bf 346} (2019), 773--804.
\bibitem{Lutwak-1975} {\rm E. Lutwak}, \textit{Dual mixed volumes}, Pacific J. Math. {\bf 58}\,(2) (1975), 531--538.


\bibitem{MP} {V.~D.~Milman and A.~Pajor}, \textit{Isotropic position
and inertia ellipsoids and zonoids of the unit ball of
a normed $n$-dimensional space}, in: Geometric Aspects
of Functional Analysis, ed. by J.~Lindenstrauss and V.~Milman,
Lecture Notes in Mathematics {\bf 1376}, Springer, Heidelberg, 1989, pp.~64--104.
\bibitem{MTY-2020} {\rm S.\ Myroshnychenko, K.\ Tatarko and V.\ Yaskin}, \textit{Unique determination of ellipsoids by their dual volumes},
Int. Math. Res. Not. IMRN (to appear).
\bibitem{Osserman}
\textrm{R.~Osserman}, \textit{Bonnesen-style isoperimetric inequalities}, Amer. Math. Monthly, {\bf 86}\,(1) (1979), 1--29.
\bibitem{Rao-book} \textrm{C.~R.~Rao}, \textit{Linear statistical inference and its applications. Second edition},
Wiley Series in Probability and Mathematical Statistics. John Wiley \& Sons, New York-London-Sydney, 1973.
\bibitem{Rivin-2007} I.~Rivin, \textit{Surface area and other measures of ellipsoids}, Adv. in Appl. Math. {\bf 39} (2007), no. 4, 409--427.

\bibitem{rusu-thesis} A.~Rusu, \textit{Determining starlike bodies by their curvature integrals}, Ph.D. Thesis, University
of South Carolina, 2008.
\bibitem{RYZ-survey} D.~Ryabogin, V.~Yaskin and A.~Zvavitch, \textit{Harmonic Analysis and Uniqueness questions in Convex Geometry}, 
Recent Advances in Harmonic Analysis and Applications, In Honor of Konstantin Oskolkov's 65th Birthday, 2013, edited by D.~Bilyk, L.~De~Carli, A.~Petukhov, A.~Stokolos and B.~D.~Wick.
\bibitem{Schmuckenschlager-1999} {\rm M.\ Schmuckenschl\"{a}ger}, \textit{An extremal property of the regular simplex},
Convex geometric analysis (Berkeley, CA, 1996), 199--202, Math. Sci. Res. Inst. Publ., 34, Cambridge Univ. Press, Cambridge, 1999.

\bibitem{Schneider-book} {R.~Schneider},  \textit{Convex bodies: The Brunn-Minkowski
theory}, 2nd expanded ed. Encyclopedia of Mathematics and its
Applications, 151. Cambridge: Cambridge University Press, (2014).

\bibitem{Yaskin} {\rm V. Yaskin}, \textit{On perimeters of sections of convex polytopes}, J. Math. Anal. Appl. {\bf 371} (2010), 447–453.

\end{thebibliography}

\medskip

\thanks{\noindent {\bf Keywords:} Convex body, volume, surface area, slicing inequality, sections and projections.

\smallskip

\thanks{\noindent {\bf 2010 MSC:} Primary 52A20; Secondary 46B06, 52A40, 52A38, 52A23.}

\bigskip

\bigskip

\noindent \textsc{Silouanos \ Brazitikos}: Department of
Mathematics, National and Kapodistrian University of Athens, Panepistimioupolis 157-84,
Athens, Greece.

\smallskip

\noindent \textit{E-mail:} \texttt{silouanb@math.uoa.gr}

\bigskip

\noindent \textsc{Dimitris-Marios \ Liakopoulos}: Department of
Mathematics, National and Kapodistrian University of Athens, Panepistimioupolis 157-84,
Athens, Greece.

\smallskip

\noindent \textit{E-mail:} \texttt{dliakop@math.uoa.gr}

\bigskip

\end{document}